\documentclass{amsart}
\usepackage{amsmath,amssymb, amscd, stmaryrd}
\usepackage{fullpage}
\usepackage{enumerate}
\usepackage{cite}
\usepackage{framed}
\usepackage[dvips,dvipdf]{graphicx}
\usepackage[usenames,dvipsnames]{color}
\usepackage{color}
\usepackage{tabularx}
\usepackage{array}
\usepackage{multirow}
\usepackage{changes}
\newtheorem{theorem}{Theorem}[section]
\newtheorem{proposition}[theorem]{Proposition}

\theoremstyle{definition}

\theoremstyle{definition}

\theoremstyle{remark}
\newtheorem{remark}[theorem]{Remark}
\newtheorem{example}[theorem]{Example}

\newcommand{\mb}[1]{\ensuremath{\mathbf{#1}}}

\newcommand{\NN}{\mathbb N}
\newcommand{\PP}{\mathbb P}
\newcommand{\QQ}{\mathbb Q}
\newcommand{\RR}{\mathbb R}
\newcommand{\ZZ}{\mathbb Z}

\newcommand{\floor}[1]{\left\lfloor#1\right\rfloor}

\author[J. Ovall]{Jeffrey S. Ovall} \address{Jeffrey S. Ovall,
  Fariborz Maseeh Department of Mathematics and Statistics, Portland
  State University, Portland, OR 97201}
\email{jovall@pdx.edu}

\author[S. Reynolds]{Samuel E. Reynolds} \address{Samuel Reynolds, Fariborz Maseeh Department of Mathematics and Statistics, Portland
  State University, Portland, OR 97201}
\email{ser6@pdx.edu}

\begin{document}
\title[Quadrature on Curvilinear Polygons]{
		Quadrature for Implicitly-defined Finite Element Functions on Curvilinear Polygons
	} 
\date{\today}

\begin{abstract}
  $H^1$-conforming Galerkin methods on polygonal meshes such as VEM,
  BEM-FEM and Trefftz-FEM employ local finite element functions that
  are implicitly defined as solutions of Poisson problems having
  polynomial source and boundary data.  Recently, such methods have
  been extended to allow for mesh cells that are curvilinear polygons.
  Such extensions present new challenges for determining suitable
  quadratures.  We describe an approach for integrating products of
  these implicitly defined functions, as well as products of their
  gradients, that reduces integrals on cells to integrals along their
  boundaries.  Numerical experiments illustrate the practical
  performance of the proposed methods.
\end{abstract}

\maketitle

\section{Introduction}\label{Intro}

The construction, analysis and implementation of finite element
methods employing meshes consisting of non-standard cell shapes
(e.g. fairly general polygons in 2D and polyhedra in 3D) have
generated a lot of interest, and a sizable literature, in the last 10+
years.  We do not attempt here to provide a representative sample of
the literature, but instead highlight three closely-related approaches
for second-order (linear) elliptic problems that yield
$H^1$-conforming finite element spaces, and provide motivation for the
problems considered in this work.  Virtual Element Methods (VEM)
(cf. \cite{BeiraodaVeiga2013,Ahmad2013,Brezzi2014,BeiraodaVeiga2014,BeiraodaVeiga2014a,Antonietti2014,Gain2014,BeiraodaVeiga2016,Benedetto2016,Antonietti2017,Antonietti2019,BeiraodaVeiga2019,BeiraodaVeiga2020}),
Boundary Element-Based Finite Element Methods (BEM-FEM)
(cf.~\cite{Weisser2011,Rjasanow2012,Rjasanow2014,Weisser2014,Hofreither2016,Weisser2017,Weisser2018,Weisser2019,Weisser2019a,Seibel2020})
and Trefftz-type Finite Element Methods (Trefftz-FEM)
(cf. \cite{Copeland2009,Hofreither2010,Hofreither2011,Anand2018,Anand2020})
all employ vector spaces that are implicitly defined, as
in~\eqref{PoissonSpace}.  The primary difference between VEM on the
one hand, and BEM-FEM and Trefftz-FEM on the other, is how they work
with such implicitly defined spaces in practice, particularly with
regard to forming finite element linear systems.  In VEM, the space is
treated ``virtually'' via degrees of freedom that provide enough
information, up to a well-chosen stabilization term, to form the
linear system; computations with basis functions (shape functions) are
avoided.  In contrast, BEM-FEM and Trefftz-FEM work more directly with
basis functions, which are defined implicitly in terms of Poisson
problems with explicitly given (polynomial) data.  All computations
involving these basis functions (e.g. pointwise evaluation of functions and some of
their derivatives) are carried out using the solution of associated
boundary integral equations. 
BEM-FEM and Trefftz-FEM
%share the same basic methodology, and
differ primarily in the types of boundary integral equations that are
solved (typically second-kind for Trefftz-FEM, and first-kind for
BEM-FEM), and the discretizations employed for solving them (Nystr\"om
methods for recent versions of Trefftz-FEM, and BEM for BEM-FEM).
Recently, Trefftz-FEM and VEM have been extended to allow for mesh
cells that are curvilinear polygons
(cf.\cite{Anand2018,BeiraodaVeiga2019,Anand2020,BeiraodaVeiga2020}).

Determining suitable quadratures for finite elements on general
polytopal meshes is clearly more challenging than for standard meshes,
such as those involving only simplices and/or basic (affine)
transformations of tensor product cells, for which polynomial-based
quadratures of high order are readily available.  Allowing for general
curvilinear polygons and non-polynomial functions further complicates
the matter.  An obvious approach to quadrature on polytopes is to
first partition it into such standard mesh cells, apply the known
quadratures on each, and sum the result.  This ``brute-force''
approach, though costly, remains popular because of the simplicity of
its implementation.  It is used, for example, in the BEM-FEM
literature, whenever higher-order spaces, e.g.~\eqref{PoissonSpace}
for $m>1$, are employed. Making no attempt at being exhaustive, we
briefly describe some of the more sophisticated approaches---the
introduction in~\cite{Antonietti2018} provides a good starting point
for a more detailed exploration.  In~\cite{Sommariva2007}, the authors
describe an approach yielding Gauss-like quadratures for polygons,
which are exact for polynomials of a given degree $2n-1$.  An
unattractive feature of the approach is that it often leads to
quadrature points that lie outside the polygon.  However, in many
cases, including all convex polygons, a modified version of their
approach yields all quadrature points in the polygon.
In~\cite{Sudhakar2014}, the authors propose a similar approach which
also allows for integration on polyhedra, and provide a more careful
reporting of its practical efficiency.
In~\cite{Natarajan2009}, the authors first use a Schwarz-Christoffel
(conformal) mapping to transform a polygon to the unit disk, which is
clearly tensorial in polar coordinates, and then any number of basic
1D quadratures (e.g. midpoint rule) may be applied in the radial and
angular directions.  The authors discuss numerical methods for
determining the conformal map for each polygon.  A more sophisticated
subpartitioning approach is described in~\cite{Talischi2014}, with the
aim of integrating non-polynomial (often rational) functions such as
several variants of ``generalized barycentric coordinates'' naturally
arising in Polygonal FEM (PFEM) (cf.~\cite{Floater2015,Hormann2018}).
The contribution~\cite{Sudhakar2017} describes methods for systematically
``compressing''  pre-existing quadrature rules, keeping a subset of the
original nodes and recomputing quadrature weights, in order to retain
as much of the effectiveness of the original quadrature while often
drastically reducing the number of quadrature nodes.
In~\cite{Antonietti2018}, the authors consider the integration of
polynomials on (flat-faced) polytopes in $\RR^d$.  A simple identity
for homogeneous functions allows recursive reduction of the integral
of a polynomial on the polytope to a sum of integrals
lower-dimensional facets, ending in evaluation at the vertices.  The
first step in their reduction can also be used for curvilinear
polygons, so we provide further details in
Section~\ref{ReductionToBoundary}.  A very recent
contribution~\cite{Artioli2020} considers quadratures on curvilinear
polygons whose curved edges 
are arcs of circles, and whose vertices are also the vertices of a
convex polygon.  This approach first partitions the cell into (curved)
triangles and rectangles, each having at most once curved edge,
generates quadratures for these specialized shapes, and then sums them
to obtain a quadrature for the entire cell.  Finally, a compression
technique is used to reduce the number of quadrature points.
%With the possible exception of the
%conformal mapping approach, the others are best suited to 

Given $m\in\ZZ$, we denote by $\PP_m$ the vector space of
(real-valued) polynomials of total degree $\leq m$ on $\RR^2$, with
the convention that $\PP_m=\{0\}$ when $m<0$.  We recall that
$\dim\PP_m=\binom{m+2}{2}$.  For non-empty $S\subset \RR^2$, we define
$\PP_m(S)$ as the restriction of $\PP_m$ to $S$, and note that it is
also a vector space.  The dimension of $\PP_m(S)$ depends on the
nature of $S$.  For example, if $S$ is open and connected, then
$\dim\PP_m(S)=\dim\PP_m=\binom{m+2}{2}$; if $S$ consists of a
single point, then $\dim\PP_m(S)=1$;  if $S$ is a straight line or a
segment thereof, then $\dim\PP_m(S)=\binom{m+2}{2}-\binom{m+1}{2}
=m+1$; and if $S$ is an arc of
%an ellipse or hyperbola,
a curved conic section,
then $\dim\PP_m(S)=\binom{m+2}{2}-\binom{m}{2}$.
Further discussion of $\PP_m(e)$, where $e$ is a simple (bounded)
curve, is given in~\cite{Anand2020}.

Let $K \subset \RR^2$ be open, bounded, simply connected, whose
Lipschitz boundary $\partial K$ is a union of smooth arcs with
disjoint interiors, which we call edges.  We refer to $K$ as a
curvilinear polygon.  Points where adjacent edges meet are called
vertices, and adjacent edges are allowed to meet at a straight angle.
For $m\in\NN$, we define $\PP_m^*(\partial K)$ to
be the vector space of continuous functions on $\partial K$ such that
the restriction of such a function to an edge $e$ of $K$ is in
$\PP_m(e)$.  It is clear that $\PP_m(\partial K)\subset \PP_m^*(\partial K)$.
We define the space $V_m(K)$ as
\begin{align}\label{PoissonSpace}
V_m(K)=\{v\in H^1(K):\, \Delta v\in\PP_{m-2}(K)\mbox{ in }K\,,\, v_{\vert_{\partial
  K}}\in\PP_m^*(\partial K)\}~.
\end{align}
It is apparent from the definition that $\PP_m(K)\subset V_m(K)$, and
it can be shown that the only way in which equality is
achieved is when $m=1$ and $K$ is a triangle.   A natural
decomposition of $V_m(K)$ is $V_m(K)=V_m^{\partial K}(K)\oplus
V_m^{K}(K)$, where
\begin{subequations}
\begin{align}\label{PoissonSpaceSplitting}
  V ^{\partial K}_m(K)&=\{v\in H^1(K):\, \Delta v=0\mbox{ in }K\,,\, v_{\vert_{\partial
  K}}\in\PP_m^*(\partial K)\}~,\\
 V ^{K}_m(K)&=\{v\in H^1(K):\, \Delta v\in\PP_{m-2}(K)\mbox{ in }K\,,\, v_{\vert_{\partial
  K}}=0\}~.
\end{align}
\end{subequations}
We see that $\dim V ^{K}_m(K)=\dim \PP_{m-2}(K)=\binom{m}{2}$, and
that $\dim V ^{\partial K}_m(K)=\dim \PP_m^*(\partial K)$, the latter
of which not only depends on $m$, but also on the number and nature of
the edges of $K$.

We consider methods for efficiently approximating integrals of the forms
\begin{align}\label{TargetIntegrals}
\int_K vw\,dx\quad,\quad \int_K\nabla  v\cdot\nabla w\,dx~,
\end{align}
for $v,w\in V_m(K)$, by reducing the computations to integrals along
the boundary $\partial K$.  As we will see in subsequent sections, we
will never need to evaluate functions or their gradients in the
interior of $K$ in order to evaluate the integrals in~\eqref{TargetIntegrals} for
$v,w\in V_m(K)$.  We will only need access to Dirichlet and Neumann
traces of associated functions.

%\begin{remark}\label{ExoticElements}
%  Cases where there are relatively few ``exotic'' mesh cells, or at
%  least relatively few ``types'' of exotic mesh cells---mesh cells are
%  of the same type if they differ only by affine transformation.
%\end{remark}

The rest of the paper is organized as follows.  In
Section~\ref{Preliminaries}, we discuss algebraic and computational
techniques for determining functions whose Laplacian is either a
polynomial or a harmonic function.  Using these, we describe how the
integrals~\eqref{TargetIntegrals} can be reduced to associated
integrals on cell boundaries.  Quadratures along the edges require
evaluation of associated functions and their normal derivatives,
which are provided algebraically for polynomials and via boundary
integral equations for harmonic functions.  Numerical experiments
illustrating the practical performance of the approach are provided in
Section~\ref{Experiments}.
Section~\ref{Details} contains further details, such as our chosen
quadrature for edges, that are not as central to the discussion and would
unnecessarily bog down the reading of the paper if they were included earlier.

%%%%%%%%%%%%%%%%%%%%%%%%%%%%%%%%%%%%%%%%%%%%%%%%%%%%%%%%%%%%%%%%%
%%%%%%%%%%%%%%%%%%%%%%%%%%%%%%%%%%%%%%%%%%%%%%%%%%%%%%%%%%%%%%%%%

\section{Preliminary Results}\label{Preliminaries}
It will be useful to have techniques for solving Poisson problems having
polynomial or harmonic source terms.
The following result is a corollary
of~\cite[Theorem 2]{Karachik2010}.
\begin{proposition}\label{Laplacep}
  Suppose $p\in\PP_n$.  There is a $P\in\PP_{n+2}$ such
  that $\Delta P=p$.  More explicitly, if $p(x)=\sum_{|\alpha|\leq
    n}c_\alpha(x-z)^\alpha$, then we may choose
  \begin{align}\label{LaplaceP}
    P(x)=\sum_{|\alpha|\leq n}c_\alpha P_\alpha(x)\quad,\quad
    P_\alpha(x)=\frac{|x-z|^2}{4 (|\alpha|+1)!}
 \sum_{k=0}^{\lfloor|\alpha|/2\rfloor}\frac{(-1)^k(|\alpha|-k)!}{(k+1)!}\left(\frac{|x-z|^2}{4}\right)^k \Delta^k(x-z)^\alpha~,
  \end{align}
  where $\lfloor s\rfloor$ is the integer part of $s$.
%  Furthermore, if $p$ is homogeneous of degree $n$, then $P$ is
%  homogeneous of degree $n+2$.
\end{proposition}
We note that $P_\alpha$, which satisfies
$\Delta P_\alpha=(x-z)^{\alpha}$, is homogeneous of degree
$|\alpha|+2$, and that these polynomials may be computed
\textit{offline} and tabulated for  $|\alpha|$ up to some specified
threshold, so that $P$ may be computed efficiently from the
coefficients $c_\alpha$ of $p$.  A polynomial $q\in\PP_j$ given with
respect to a shifted monomial basis, as is $p$ above, can be encoded
as a coefficient array of length $J=\binom{j+2}{2}$, once a suitable
enumeration of the multiindices $\alpha$ is chosen.  Such an
enumeration is given in Table~\ref{Palpha} for $|\alpha|\leq 6$,
and its extension to all multiindices is clear.  
%%%%%%%%%%%%%%%%%%%%%%%%%%%%%
\iffalse
\begin{table}
  \centering
  \caption{\label{MultiIndex} An enumeration of multiindices $\alpha$,
  for $|\alpha|\leq 4$.}
\begin{tabular}{|ccc|ccc|ccc|}\hline
  $k$&$\alpha$&$|\alpha|$&$k$&$\alpha$&$|\alpha|$&$k$&$\alpha$&$|\alpha|$\\\hline
  0&$(0,0)$&0&5&$(0,2)$&2&10&$(4,0)$&4\\
  1&$(1,0)$&1&6&$(3,0)$&3&11&$(3,1)$&4\\
  2&$(0,1)$&1&7&$(2,1)$&3&12&$(2,2)$&4\\
  3&$(2,0)$&2&8&$(1,2)$&3&13&$(1,3)$&4\\
  4&$(1,1)$&2&9&$(0,3)$&3&14&$(0,4)$&4\\\hline
\end{tabular}
\end{table}
\fi
%%%%%%%%%%%%%%%%%%%%%%%%%%%%%
\begin{table}
  \centering
  \caption{\label{Palpha} An enumeration of multiindices for
    $|\alpha|\leq 6$, together with the indices and values of the
    non-zero coefficients of $P_\alpha$ from Proposition~\ref{Laplacep}.}
\begin{tabular}{|cccc|}\hline
  $k$&$\alpha$&$P_\alpha$ indices&$P_\alpha$ coefficients\\\hline
  0&$(0,0)$&$(3,5)$&$(1,1)/4$\\\hline
  1&$(1,0)$&$(6,8)$&$(1,1)/8$\\
  2&$(0,1)$&$(7,9)$&$(1,1)/8$\\\hline
  3&$(2,0)$&$(10,12,14)$&$(7,6,-1)/96$\\
  4&$(1,1)$&$(11,13)$&$(1,1)/12$\\
  5&$(0,2)$&$(10,12,14)$&$(-1,6,7)/96$\\\hline
  6&$(3,0)$&$(15,17,19)$&$(3,2,-1)/64$\\
  7&$(2,1)$&$(16,18,20)$&$(11,10,-1)/192$\\
  8&$(1,2)$&$(15,17,19)$&$(-1,10,11)/192$\\
  9&$(0,3)$&$(16,18,20)$&$(-1,2,3)/64$\\\hline
  10&$(4,0)$&$(21,23,25,27)$&$(31,15,-15,1)/960$\\
  11&$(3,1)$&$(22,24,26)$&$(13,10,-3)/320$\\
  12&$(2,2)$&$(21,23,25,27)$&$(-1,15,15,-1)/360$\\
  13&$(1,3)$&$(22,24,26)$&$(-3,10,13)/320$\\
  14&$(0,4)$&$(21,23,25,27)$&$(1,-15,15,31)/960$\\\hline
  15&$(5,0)$&$(28,30,32,34)$&$(9,3,-5,1)/384$\\
  16&$(4,1)$&$(29,31,33,35)$&$(57,35,-21,1)/1920$\\
  17&$(3,2)$&$(28,30,32,34)$&$(-3,63,55,-11)/1920$\\
  18&$(2,3)$&$(29,31,33,35)$&$(-11,55,63,-3)/1920$\\
  19&$(1,4)$&$(28,30,32,34)$&$(1,-21,35,57)/1920$\\
  20&$(0,5)$&$(29,31,33,35)$&$(1,-5,3,9)/384$\\\hline
  21&$(6,0)$&$(36,38,40,42,44)$&$(127,28,-70,28,-1)/7168$\\
  22&$(5,1)$&$(37,39,41,43)$&$(15,7,-1,1)/672$\\
  23&$(4,2)$&$(36, 38, 40, 42, 44)$&$(-99, 2772, 2030, -812,
                                     29)/107520$\\
  24&$(3,3)$&$(37, 39, 41, 43)$&$(-1, 7, 7, -1)/280$\\
  25&$(2,4)$&$(36, 38, 40, 42, 44)$&$(29, -812, 2030, 2772,
                                     -99)/107520$\\
  26&$(1,5)$&$(37, 39, 41, 43)$&$(1, -7, 7, 15)/672$\\
  27&$(0,6)$&$(36, 38, 40, 42, 44)$&$(-1, 28, -70, 28, 127)/7168$\\\hline
\end{tabular}
\end{table}
The mappings $\alpha\mapsto k$ and $k\mapsto\alpha$ corresponding to
Table~\ref{Palpha} are
\begin{align}\label{MultiIndexMaps}
  \alpha\mapsto \binom{|\alpha|+1}{2}+\alpha_2\quad,\quad
  k\mapsto \left(|\alpha|-k+\binom{|\alpha|+1}{2}\,,\,
  k-\binom{|\alpha|+1}{2}\right)
  \mbox{ where }|\alpha|=\left\lfloor\frac{\sqrt{8k+1}-1}{2}\right\rfloor~.
\end{align}
\iffalse
We search for the largest $m$ such that $k\geq\binom{m+1}{2}$,
i.e. the largest integer solution $m$ of $m^2+m-2k\leq 0$.  The larger
of the two roots of $t^2+t-2k=0$ is $t=(-1+\sqrt{1+8k})/2$, which we
round down to the nearest integer $m$.  Note that $m=|\alpha|$.
\fi
Basic procedures such as computing the product of two polynomials, or
computing the gradient of a polynomial, both of which are needed for
the integral computations discussed in Section~\ref{ReductionToBoundary}, can
be performed efficiently in terms of coefficient arrays.  For
``sparse'' polynomials, i.e. those that involve relatively few
non-zero coefficients, significant efficiency can be gained by storing
only the non-zero coefficients and their indices.
The polynomials $P_\alpha$ from Proposition~\ref{Laplacep} are also
given in Table~\ref{Palpha} for $|\alpha|\leq 6$,
expressed in terms their (few) non-zero coefficients and the indices
of these coefficients.  For example, we see that, for $\alpha=(2,3)$,
$P_\alpha$ is a linear combination of the shifted monomials associated
with multiindices $29\mapsto (6,1)$, $31\mapsto (4,3)$, $33\mapsto
(2,5)$ and $35\mapsto(0,7)$.  More specifically,
\begin{align*}
P_\alpha(x)=\frac{1}{1920}\left(-11(x-z)^{(6,1)}+55(x-z)^{(4,3)}+63(x-z)^{(2,5)}-3(x-z)^{(0,7)}\right)\mbox{
  for }\alpha=(2,3)~.
\end{align*}
Though there are several patterns in the non-zero coefficients of
$P_\alpha$ that may be of interest (and some that might be exploited), we
highlight only one: the alternating (in sign) sum of the coefficients
of $P_\alpha$ is $0$.  For example,
$(-11+(-1)(55)+63+(-1)(-3))/1920=0$, for $\alpha=(2,3)$.
An extension of Table~\ref{Palpha}, for $7\leq |\alpha|\leq 10$, is
given in Table~\ref{Palpha2} in Section~\ref{Details} for convenience.

We recall that, for any function $\phi$ that is harmonic in $K$, there
is a \textit{harmonic conjugate} $\hat\phi$, satisfying
$\Delta\hat\phi=0$ and the Cauchy-Riemann equations,
\begin{align}\label{CauchyRiemann}
\frac{\partial\phi}{\partial x_1}=\frac{\partial\hat\phi}{\partial
  x_2}\quad,\quad
\frac{\partial\phi}{\partial x_2}=-\frac{\partial\hat\phi}{\partial  x_1}~,  
\end{align}
in $K$.
Such a harmonic conjugate is unique, up to an additive constant, and
the orthogonality of $\nabla\phi$ and $\nabla\hat\phi$ implies that
\begin{align}\label{ComplementaryBC}
\frac{\partial\phi}{\partial\mb{n}}=\frac{\partial\hat\phi}{\partial\mb{t}}\quad,\quad
  \frac{\partial\hat\phi}{\partial\mb{n}}=-\frac{\partial\phi}{\partial\mb{t}}
  \mbox{ on }\partial K~,
\end{align}
where $\mb{n}$ is the outward unit normal, and $\mb{t}$ is the unit
tangent in the counter-clockwise direction.  We can compute a
harmonic conjugate $\hat\phi$ by solving the Neumann problem
\begin{align}\label{HarmConjNeum}
  \Delta\hat\phi=0\mbox{ in }K\quad,\quad \frac{\partial\hat\phi}{\partial\mb{n}}=-\frac{\partial\phi}{\partial\mb{t}}
  \mbox{ on }\partial K\quad,\quad \int_{\partial K}\hat\phi\,ds=0~.
\end{align}
The condition $\int_{\partial K}\hat\phi\,ds=0$ ensures that there is
a unique solution of~\eqref{HarmConjNeum}.

We are now ready to describe an approach to computing a function whose
Laplacian is a given harmonic function.
\begin{proposition}\label{Laplacephi}
Suppose $\phi$ is harmonic in $K$.  The following construction
provides a function $\Phi$ such that $\Delta \Phi=\phi$ in $K$.
\begin{enumerate}
\item Determine a solution of the Neumann problem:  $\Delta \hat\phi=0$ in
  $K$, $\partial \hat\phi/\partial\mb{n}=-\partial\phi/\partial\mb{t}$ on
  $\partial K$. 
  % \item Determine a harmonic conjugate $\hat\phi$ of $\phi$.
\item Determine a solution of the Neumann problem:  $\Delta\rho=0$ in
  $K$, $\partial\rho/\partial \mb{n}=(\phi,-\hat\phi)\cdot \mb{n}$ on
  $\partial K$. 
\item Determine a solution of the Neumann problem:  $\Delta\hat\rho=0$ in
  $K$, $\partial\hat\rho/\partial \mb{n}=(\hat\phi,\phi)\cdot \mb{n}$ on
  $\partial K$.
\item Set $\Phi=(x_1\rho+x_2\hat\rho)/4$.
\end{enumerate}
It holds that $\phi$ and $\hat\phi$ are harmonic conjugates, and that
$\rho$ and $\hat\rho$ are harmonic conjugates.
\end{proposition}
\begin{proof}
  It follows from the Cauchy-Reimann equations~\eqref{CauchyRiemann} that both of the vector
  fields $(\phi,-\hat\phi)$ and $(\hat\phi,\phi)$ are conservative in $K$, so
  there are functions $\rho$ and $\hat\rho$ such that
  $\nabla\rho=(\phi,-\hat\phi)$ and $\nabla\hat\rho=(\hat\phi,\phi)$
  in $K$.  The potentials $\rho$ and $\hat\rho$ are unique, up to
  additive constants.
  We note that $\rho$ and $\hat\rho$ also satisfy the
  Cauchy-Reimann equations in $K$,
  \begin{align*}
     \frac{\partial \rho}{\partial x_1}=\frac{\partial \hat\rho}{\partial x_2}=\phi\quad,\quad
  \frac{\partial \rho}{\partial x_2}=-\frac{\partial \hat\rho}{\partial x_1}=-\hat\phi~.
  \end{align*}
  It also follows from the Cauchy-Reimann equations for $\phi$,
  $\hat\phi$ that $\Delta\rho=\Delta\hat\rho=0$ in $K$.  In other
  words, $\rho$ and $\hat\rho$ are also harmonic conjugates.   By
  continuously extending their gradients to the boundary, we see that
  their normal derivatives must satisfy
  \begin{align*}
    \frac{\partial\rho}{\partial
    \mb{n}}=\begin{pmatrix}\phi\\-\hat\phi\end{pmatrix}\cdot
    \mb{n}\quad,\quad
    \frac{\partial\hat\rho}{\partial
    \mb{n}}=\begin{pmatrix}\hat\phi\\\phi\end{pmatrix}\cdot
    \mb{n} 
  \end{align*}
  on $\partial K$, which leads to the two Neumann problems given in
  the proposition.  Finally, we have
  \begin{align*}
    \Delta\Phi=\frac{1}{4}\left(2 \frac{\partial \rho}{\partial x_1}+2
    \frac{\partial \hat\rho}{\partial x_2}\right)=\phi\;\mbox{ in K}~,
  \end{align*}
  which completes the proof.
\end{proof}

\begin{remark}
  Taking $\phi$, $\hat\phi$, $\rho$ and $\hat\rho$ as in
  Proposition~\ref{Laplacephi},
  we have $\Delta\widehat\Phi=\hat\phi$ for $\widehat\Phi=(x_1\hat\rho-x_2\rho)/4$.
\end{remark}

The approach described in Proposition~\ref{Laplacephi} involves the
solution of three Neumann problems, each of which can be made
well-posed by imposing the vanishing boundary integral condition as
in~\eqref{HarmConjNeum}.  There are many well-established techniques
involving boundary integral equations for solving such Neumann
problems, and we will use that described in~\cite{Ovall2018}, which
employs Nystr\"om discretizations of (well-conditioned) second-kind
integral equations.  We mention a few relevant features of the
approach in~\cite{Ovall2018} for the conjugate pair $(\phi,\hat\phi)$:
\begin{itemize}
\item $\phi$ is given implicitly in terms of its Dirichlet trace on
  $\partial K$, and supplies the boundary data for the Neumann problem
  for $\hat\phi$ via $\partial\hat\phi/\partial\mb{n}=-\partial\phi/\partial\mb{t}$.
\item A Neumann-to-Dirichlet map,
  $\partial\hat\phi/\partial\mb{n}\mapsto \hat\phi$ on $\partial K$, is
  obtained directly as the solution of a boundary integral equation.
\item A Dirichlet-to-Neumann map, $\phi\mapsto
  \partial\phi/\partial\mb{n}$ on $\partial K$, is then
  obtained from $\hat\phi$ via taking its tangential derivative, $\partial\phi/\partial\mb{n}=\partial\hat\phi/\partial\mb{t}$.
\end{itemize}
In each step, all computations occur only on $\partial K$.

%%%%%%%%%%%%%%%%%%%%%%%%%%%%%%%%%%%%%%%%%%%%%%%%%%%%%%%%%%%%%%%%%
%%%%%%%%%%%%%%%%%%%%%%%%%%%%%%%%%%%%%%%%%%%%%%%%%%%%%%%%%%%%%%%%%

\section{Reducing volumetric integrals to boundary integrals}\label{ReductionToBoundary}

We begin with the integration of polynomials on $K$.  Let
$r\in\PP_n(K)$.  In the spirit of Proposition~\ref{Laplacep}, we can
reduce $\int_Kr\,dx$ to an integral along $\partial K$, by taking
$R\in\PP_{n+2}(K)$ such that $\Delta R=r$.  It then follows that
$\int_Kr\,dx=\int_{\partial K}\partial{R}/\partial\mb{n}\,ds$.
Alternatively, we have a reduction to the boundary based on the
Divergence Theorem and the following simple identity, which can be
verified by direct computation,
\begin{align*}
\nabla\cdot[(x-z)^\alpha\,(x-z)]=(2+|\alpha|)(x-z)^\alpha~.
\end{align*}
Here and following, $z\in\RR^2$ may be chosen arbitrarily---the
barycenter of $K$ is a natural choice.
From this, it follows that
\begin{align}\label{PolyInt}
\int_{K}(x-z)^\alpha\,dx=\frac{1}{2+|\alpha|}\int_{\partial K}(x-z)^\alpha\,(x-z)\cdot\mb{n}\,ds~.
\end{align}
This type of reduction is the core of the method described
in~\cite{Antonietti2018}, where $K$ is a polytope in $\RR^d$.  In this
case $2+|\alpha|$ is replaced by $d+|\alpha|$ above, and further
reductions of the same type can be made due to the fact that the faces
of $K$ are flat.  If $F$ is such a flat face (we would take an edge $e$ in
our case), then $(x-z)\cdot\mb{n}(x)$ is constant for $x\in F$.  More
specifically, for $x\in F$, $(x-z)\cdot\mb{n}(x)$ is the signed distance between $z$
and the hyperplane containing $F$, taking the positive sign if $z\in K$.
Factoring out this constant, the integral on $F$ can be further
reduced to integrals along its $(d-2)$-dimensional (flat) facets, and
so on.   For generic curved boundaries, further simple reductions of
$\int_{\partial K}(x-z)^\alpha\,(x-z)\cdot\mb{n}\,ds$ are not
available.  Regardless,~\eqref{PolyInt} allows for the
efficient evaluation of $\int_K r\,dx$ in terms of integrals along
$\partial K$.
%%%%%
\iffalse
\begin{align*}
  r(x)=\sum_{|\alpha|\leq n}c_\alpha(x-z)^\alpha\quad,\quad
  \tilde{r}(x)=\sum_{|\alpha|\leq n}\frac{c_\alpha}{2+|\alpha|}\,(x-z)^\alpha~.
\end{align*}
\begin{align*}
  \int_K r\,dx=\int_{\partial K}\tilde{r}\, (x-z)\cdot\mb{n}\,ds~.
\end{align*}
\fi
We opt for the approach based on~\eqref{PolyInt}, as opposed to that
based on $\int_Kr\,dx=\int_{\partial
  K}\partial{R}/\partial\mb{n}\,ds$, because it is a bit cheaper.

The integration of $\int_K \nabla v\cdot\nabla w\,dx$ can be naturally
considered in three cases, the first two of which involve at least one
function from $V^{\partial K}_m(K)$.  These easier two integrals are
reduced to
\begin{subequations}\label{GradientIntegrals}
\begin{align}
  \int_K \nabla v\cdot\nabla w\,dx&=0\mbox{ when }
                                    v\in  V^{\partial K}_m(K)\,,\, w\in  V ^{K}_m(K)~,\label{Grad1}\\
  \int_K \nabla v\cdot\nabla w\,dx&=
                                    \int_{\partial K}\frac{\partial
                                    v}{\partial \mb{n}}\,w\,ds\mbox{ when }
                                    v,w\in  V^{\partial K}_m(K)~.\label{Grad2}
\end{align}
\end{subequations}
In the case of~\eqref{Grad2}, both $v$ and $w$ are given (implicitly)
in terms of their Dirichlet data, so a Dirichlet-to-Neumann map,
$v\mapsto \partial v/\partial\mb{n}$ on $\partial K$, is needed to
evaluate the boundary integral.  This can be done as discussed in
Section~\ref{Preliminaries}, or by some other method of choice.

Now suppose that $v,w\in  V^{K}_m(K)$.  These are given (implicitly) in terms of
$p,q\in\PP_{m-2}(K)$ such that $\Delta v=p$ and $\Delta w=q$ in $K$.
Let $P,Q\in\PP_m(K)$ be such that $\Delta P=p$ and $\Delta Q=q$.  We have
\begin{align*}
  \int_K \nabla v\cdot\nabla w\,dx&=\int_K \nabla v\cdot\nabla Q\,dx+
                                    \int_K \nabla v\cdot\nabla
                                    (w-Q)\,dx\\
                                  &=\int_K \nabla
                                    v\cdot\nabla Q\,dx=\int_{\partial
                                    K}\frac{\partial v}{\partial\mb{n}}\,Q\,ds-\int_KpQ\,dx~.
\end{align*}
%The integral $\int_K \nabla v\cdot\nabla (w-P)\,dx$ vanishes because
%$w-P$ is harmonic in $K$ and $v$ vanishes on $\partial K$, i.e. for
%the same reasons that~\eqref{Grad1} holds.
Summarizing, we have
\begin{align}\label{Grad3}
  \int_K \nabla v\cdot\nabla w\,dx&=\int_{\partial K}\frac{\partial v}{\partial\mb{n}}\,Q\,ds-
                                    \int_{K}pQ\,dx\mbox{ when }v,w\in  V^{K}_m(K)~.
\end{align}
Since $pQ\in\PP_{2m-2}(K)$, the integral $\int_{K}pQ\,dx$ can be
adresses as discussed at the beginning of this section.
The only
term in the boundary integrals in~\eqref{Grad3} that requires further
consideration is $\partial v/\partial\mb{n}$.  Unlike~\eqref{Grad2},
$v$ is not harmonic in this case, so an additional step is needed to
determine $\partial v/\partial\mb{n}$.   This can be done as follows,
\begin{align}\label{NeumannTraceHardCase}
  \frac{\partial v}{\partial \mb{n}}=\frac{\partial (v-P)}{\partial\mb{n}}+\frac{\partial P}{\partial\mb{n}}~.
\end{align}
The term $\partial P/\partial\mb{n}$ can be computed directly from the
known polynomial $P$, and the term $\partial (v-P)/\partial\mb{n}$ can be computed
using a Dirichlet-to-Neumann map as discussed earlier, because $v-P$
is harmonic, with known boundary trace, $-P$.

The integral $\int_K vw\,dx$ is more challenging than its gradient
counterpart, and we do not bother splitting into cases as before.  To
fix notation,
\begin{align*}
\Delta v=p\in\PP_{m-2}(K)\mbox{ in }K\quad,\quad
  v=f\in\PP^*_m(K)\mbox{ on }\partial K~,\\
  \Delta w=q\in\PP_{m-2}(K)\mbox{ in }K\quad,\quad
  w=g\in\PP^*_m(K)\mbox{ on }\partial K~.
\end{align*}
As above, $P,Q\in\PP_m(K)$ satisfy $\Delta P=p$ and $\Delta Q=q$.  We
have
\begin{align*}
  \int_Kvw\,dx=\int_K(v-P)(w-Q)\,dx+\int_K(v-P)Q\,dx+\int_KP(w-Q)\,dx+\int_KPQ\,dx~.
\end{align*}
%Since $PQ\in\PP_{2m}(K)$ is known, take $R\in\PP_{2m+2}(K)$ such that
%$\Delta R=PQ$.
Now, take $P^*,Q^*\in\PP_{m+2}(K)$ such that
$\Delta P^*=P$ and $\Delta Q^*=Q$, and $\Phi$ such that
$\Delta\Phi=v-P$ in $K$, as indicated in Proposition~\ref{Laplacephi}.
At this stage, we have
\begin{align*}
  \int_Kvw\,dx=\int_K\Delta\Phi\,(w-Q)\,dx+\int_K(v-P)\Delta
  Q^*\,dx+\int_K\Delta P^*\,(w-Q)\,dx+\int_KPQ\,dx~.
\end{align*}
The integrands in four of these integrals are the product of a harmonic function with
the Laplacian of a second function.  Using Green's identities to move
the Laplacian over to the harmonic function, we obtain
\begin{align}\label{L2case}
  \begin{split}
  \int_Kvw\,dx&=
  \int_{\partial K}\frac{\partial\Phi}{\partial\mb{n}}\,(g-Q)-
  \Phi\,\frac{\partial(w-Q)}{\partial\mb{n}}\,ds+
 \int_{\partial K}\frac{\partial Q^*}{\partial\mb{n}}\,(f-P)-
                Q^*\,\frac{\partial(v-P)}{\partial\mb{n}}\,ds\\
              &+
\int_{\partial K}\frac{\partial P^*}{\partial\mb{n}}\,(g-Q)-
                P^*\,\frac{\partial(w-Q)}{\partial\mb{n}}\,ds
+\int_{K}PQ\,dx~.
\end{split}
\end{align}
As with~\eqref{Grad3}, we handle the polynomial integral as discussed
at the beginning of this section.
In the case that $w=r\in\PP_n(K)$, we have a much simpler formula for
reducing the integral to the boundary.  Let $R\in\PP_{n+2}(K)$ be such
that $\Delta R=r$.  We have
\begin{align}\label{PolynomialSourceTerm}
  \int_Kvw\,dx&=\int_{\partial K}(f-P)\,\frac{\partial R}{\partial \mb{n}}-R\,\frac{\partial (v-P)}{\partial \mb{n}}\,ds+\int_KPr\,dx~.
\end{align}
This simpler formula may be convenient for integrating basis function
against polynomial source terms in the formation of the righthand side
(load vector) for the finite element system, for example.
A different simplification of~\eqref{L2case} may be given when
both $v$ and $w$ are harmonic.  In this case,
%Furthermore, in the case that both $v$ and $w$ are harmonic,
we may
take $P=Q=P^*=Q^*=0$ in~\eqref{L2case}, and the formula reduces to
	\begin{align}
		\label{L2caseHarmonic}
		\int_K vw~dx
		=
		\int_{\partial K} \dfrac{\partial\Phi}{\partial\mathbf{n}}g~ds
		-\int_{\partial K}\Phi\dfrac{\partial w}{\partial\mathbf{n}}~ds~.
%		~, \quad
%		\text{whenever}~\Delta v=\Delta w=0~\text{in}~K
%		~,
	\end{align}
        \begin{remark}\label{NormalDerivPhi}  Both~\eqref{L2case} and
          its special case~\eqref{L2caseHarmonic} involve the normal
          derivative of the function $\Phi$ of Proposition~\ref{Laplacephi}.
          Taking $\phi=v-P$ as above, and using the notation of
          Proposition~\ref{Laplacephi}, this normal derivative is
          given by
          \begin{align*}
            \frac{\partial\Phi}{\partial\mb{n}}=\begin{pmatrix}\rho+x_1\phi+x_2\hat\phi\\\hat\rho-x_1\hat\phi+x_2\phi\end{pmatrix}\cdot\frac{\mb{n}}{4}~.
          \end{align*}
          The Dirichlet data of $v$ is given, and $P$ is readily
          obtained from $p$ based on Proposition~\ref{Laplacep} and
          look-up tables such as Table~\ref{Palpha}, so we have easy
          access to the Dirichlet data of $\phi$.  The functions
          $\hat\phi$, $\rho$ and $\hat\rho$ are all solutions of
          Neumann problems, and it is clear that we only need their
          Dirichlet data to evaluate $\partial\Phi/\partial\mb{n}$.
          The integral equations~\cite{Ovall2018} (see also
          Section~\ref{Details}) that we employ provide such
          Neumann-to-Dirichlet maps directly.
        \end{remark}

%%%%%%%%%%%%%%%%%%%%%%%%%%%%%%%%%%%%%%%%%%%%%%%%%%%%%%%%%%%%%%%%%
%%%%%%%%%%%%%%%%%%%%%%%%%%%%%%%%%%%%%%%%%%%%%%%%%%%%%%%%%%%%%%%%%

\section{Numerical Illustrations}\label{Experiments}
We illustrate the performance of our scheme to compute the target
integrals~\eqref{TargetIntegrals} by reducing them to boundary
integrals, as described in Section~\ref{ReductionToBoundary} and
highlighted by the formulas~\eqref{Grad3} and~\eqref{L2case}, and
special cases such as~\eqref{L2caseHarmonic}.  Further details on the
boundary integral equation techniques used to compute
Dirichlet-to-Neumann and Neumann-to-Dirichlet used in our approach
will be given in Section~\ref{Details}, where some discussion of the
underlying quadrature(s) will also be provided.  Here we merely state
that these quadratures are governed by two parameters $n$ and $\sigma$,
where $n$ dictates the number of quadrature points used on each edge
of $\partial K$, and $\sigma$ determines the ``strength'' of a
change-of-variable used to define the quadrature.  In the following experiments, we fix $\sigma=7$, and
vary $n$ to illustrate rapid convergence with respect to this parameter. 

\begin{example}[Constant Functions]
	\label{ConstantFunctionExample}
	Consider the case where $v,w\in H^1(K)$ are harmonic and have
        a constant boundary trace $v|_{\partial K}=w|_{\partial K}=1$.
        Clearly, $v=w=1$ in $K$, and $\int_K vw~dx = |K|$.  As a first
        basic test of our quadrature approach, we compare the computed
        value of $\int_K vw~dx$ to $|K|$ in Table
        \ref{ConstantFunctionTable} for three different cases: the
        unit square ($|K|=1$), the unit circle ($|K|=\pi$), and the
        puzzle piece described in Example~\ref{PuzzlePieceExample}
        ($|K|=1$).  We treat the circle as having two edges, with
        the two vertices at opposite ends of a diameter.

        In this special case of constant functions, if we were to follow the
        construction described in Proposition~\ref{Laplacephi} to
        obtain a $\Phi$ such that $\Delta\Phi=v=1$ ``by hand'', we
        could take $\hat{v}=0$ as the harmonic conjugate of $v$, and
        obtain $\rho=x_1$ and $\hat\rho=x_2$, ultimately yielding the
        familiar $\Phi=|x|^2/4$.  Using this $\Phi$, we see
        that~\eqref{L2caseHarmonic} reduces to
        $|K|=(1/2)\int_{\partial K}x\cdot\mb{n}\,ds$, which can also
        be seen as a special case of~\eqref{PolyInt}, with
        $\alpha=(0,0)$ and $z=(0,0)$.  However, instead of
        approximating $(1/2)\int_{\partial K}x\cdot\mb{n}\,ds$
        directly, we proceed with the approach described
        by~\eqref{L2caseHarmonic}, which only ``knows'' that $v$ and
        $w$ are harmonic and have given Dirichlet data.
        Table~\ref{ConstantFunctionTable} records the absolute errors
        in our quadrature approximations, and exhibits rapid
        convergence with respect the parameter $n$ governing the
        number of quadrature points used on each edge.
	\begin{table}[]
		\centering
		\caption{Absolute errors in the approximation of $|K|$
                  via \eqref{L2caseHarmonic} in
                  Example~\ref{ConstantFunctionExample}, for which
                  $\Delta v=\Delta w=0$ in $K$ and $v=w=1$ on
                  $\partial K$.}
%			where $v,w$ are harmonic with boundary traces identically one.
%			For the quadrature, values are sampled at $2n$ points on each edge,
%			for a total of $8n$ sampled points.
%			The Kress parameter $p_0=7$ was used throughout.}
		\label{ConstantFunctionTable}
		\begin{tabular}{|cccc|}
			\hline 	$n$		&unit square	&unit circle&puzzle piece\\ \hline
			     4	&6.2674e-03	&1.1254e-02  &1.2107e-03\\
			     8	&1.4776e-05   &1.8674e-06  &7.5746e-06\\
			    16	&1.0118e-07	&2.4451e-09  &3.3861e-07\\
			    32	&1.1940e-10	&8.9906e-12  &5.4846e-11\\
			    64	&6.2350e-13	&2.9310e-14  &1.3824e-12\\
			\hline
		\end{tabular}
              \end{table}
\end{example}

\begin{example}[Unit Square]\label{UnitSquareExample}
  Let $K=(0,1)\times(0,1)$ be the unit square.  It holds that
  $\dim V_m(K)=m(m+7)/2$.  We make a brief comparison with the tensor
  product polynomials of degree $\leq m$ in each variable,
  $\QQ_m=\mathrm{span}\{x^iy^j:\,0\leq i,j\leq m\}$, before testing
  our quadratures.  Let $\QQ_m(K)$ denote the restriction of $\QQ_m$
  to $K$.  We have $\dim\QQ_m(K)=(m+1)^2>\dim V_m(K)$ for $m\geq 3$.
  It holds that $V_1(K)=\QQ_1(K)$.  The fact that
  $x(1-x)y(1-y)\in\QQ_2(K)\setminus V_2(K)$ shows that
  $V_2(K)\neq\QQ_2(K)$, although $\dim V_2(K)=\dim\QQ_2(K)=9$.

  We will approximate several entries of the element mass and
  stiffness matrices associated with a basis of $V_2(K)$.  The
  standard basis for $V_1(K)=V_1^{\partial K}(K)=\QQ_1(K)$ consists of
  the four ``vertex functions'',
\begin{align*}
  (v_0\,,\,v_1\,,\,v_2\,,\,v_3)=((1-x)(1-y)\,,\,x(1-y)\,,\,xy\,,\,(1-x)y)~,
\end{align*}
so called because $\Delta v_j=0$ and $v_j(z_i)=\delta_{ij}$ for
$0\leq i,j\leq 3$, where
\begin{align*}
  (z_0\,,\,z_1\,,\,z_2\,,\,z_3)=\left((0,0)\,,\,(1,0)\,,\,(1,1)\,,\,(0,1)\right)
\end{align*}
are the vertices given in counter-clockwise order.  We add to this
four ``edge functions'', satisfying $\Delta w_j=0$ in $K$ and
$w_j=v_jv_{j+1}$ on $\partial K$.
Here, and elsewhere in this example, all subscripts should be understood modulo $4$, e.g. $v_4=v_0$,
$v_5=v_1$, $v_6=v_2$. 
The vertex and edge functions
together form a hierarchical basis for $V_2^{\partial K}(K)$.  We
complete a basis for $V_2(K)$ by including the function satisfying
$\Delta \tilde{w}=-1$ in $K$ and $\tilde{w}=0$ on $\partial K$.
Using separation of variables, we can obtain series expansions of
$w_j$ and $\tilde{w}$.  For example,
\begin{align*}
  w_1=\sum_{k\in2\NN-1}^\infty\frac{8 \sinh(k\pi
       x)\sin(k\pi y)}{(k\pi)^3 \sinh(k\pi)}\quad,\quad
  \tilde{w}=\sum_{k,\ell\in2\NN-1}^\infty\frac{16 \sin(k\pi x) \sin(\ell\pi y)}{k\ell(k^2+\ell^2)\pi^4}~.
\end{align*}
\iffalse
Four Edge Functions
\begin{align*}
&4\sum_{k=1}^\infty\frac{1-(-1)^k}{(k\pi)^3}\frac{\sinh(k\pi
     x)}{\sinh(k\pi)}\,\sin(k\pi y)\\
&4\sum_{k=1}^\infty\frac{1-(-1)^k}{(k\pi)^3}\frac{\sinh(k\pi
     y)}{\sinh(k\pi)}\,\sin(k\pi x)\\  
&4\sum_{k=1}^\infty\frac{1-(-1)^k}{(k\pi)^3}\frac{\sinh(k\pi(1-
     x))}{\sinh(k\pi)}\,\sin(k\pi y) \\
&4\sum_{k=1}^\infty\frac{1-(-1)^k}{(k\pi)^3}\frac{\sinh(k\pi
     (1-y))}{\sinh(k\pi)}\,\sin(k\pi x)~.
\end{align*}
\fi
Using these formulas, we obtain reference values for the desired
integrals that are either exact, or obtained to very high precision
from series expansions.  For the approximated reference values, i.e. those that are given in
decimal form, all digits are correct up to rounding in the final
digit---{\sc Mathematica} was used to compute them, employing very
high precision arithmetic.
These reference
values are used to test our quadrature on several combinations of the
basis functions.  Table~\ref{UnitSquareTableB} provides convergence
data for these tests, as well as the reference values used to compute
the errors.
\begin{table}[]
		\centering
		\caption{Absolute errors for computing the $L^2$ inner
                  product $\int_K vw~dx$ and the $H^1$ semi-inner
                  product $\int_K\nabla v\cdot\nabla w~dx$ on the unit
                  square, as described in Example
                  \ref{UnitSquareExample}. Following the notation in
                  that example, $v_j$ denotes the ``vertex''
                  functions, $w_j$ denotes the ``edge'' functions, and
                  $\tilde{w}$ denotes the ``bubble'' function.}
		\label{UnitSquareTableB}
\begin{tabular}{|ccccl|}\hline
Functions &$n$ &$L^2$ error	&$H^1$ error&Reference Values\\\hline
$v_j\;,\;v_j$
&4  &1.8197e-03&6.9331e-03&$L^2$: 1/9\\
&8  &5.0843e-06&3.6484e-05&$H^1$: 2/3\\
&16&3.3700e-08&1.1758e-07&\\
&32&4.4464e-11&1.1843e-10&\\
&64&2.4278e-13&6.5759e-13&\\
  \hline
$v_j\;,\;v_{j\pm 1}$  
&4  &8.3471e-04&8.0181e-05&$L^2$: 1/18\\
&8  &6.3177e-07&7.4406e-06&$H^1$: $-1/6$\\
&16&2.6840e-09&1.8098e-08&\\
&32&4.7440e-12&4.0427e-12&\\
&64&5.2902e-14&8.5895e-13&\\  
  \hline
$v_j\;,\;v_{j+2}$
&4  &2.7437e-04&7.5354e-03&$L^2$: 1/36\\
&8  &4.6195e-06&2.1527e-05&$H^1$: $-1/3$\\
&16&2.1823e-08&8.1290e-08&\\
&32&2.3449e-11&1.1009e-10&\\
&64&1.0834e-13&4.6124e-13&\\
\hline
$v_0\;,\;w_1$
&4  &1.4790e-06&9.6344e-04&$L^2$: 6.069682826514464e-03\\
&8  &1.2707e-06&6.1960e-06&$H^1$: $-1/12$\\
&16&6.8236e-09&3.1021e-08&\\
&32&6.8066e-12&4.5776e-11&\\
&64&2.3823e-14&4.1675e-14&\\
\hline
$v_1\;,\;w_1$
&4  &5.1158e-04&1.5100e-03&$L^2$: 1.802485697075799e-02\\
&8  &6.5354e-06&6.2160e-06&$H^1$: 1/12\\
&16&9.6573e-09&3.1038e-08&\\
&32&1.1113e-11&4.5842e-11&\\
&64&8.9987e-14&6.6937e-13&\\
\hline
$w_j\;,\;w_j$
&4  &1.6966e-04&2.0778e-03&$L^2$: 5.195037581961447e-03\\
&8  &2.7239e-06&3.6914e-05&$H^1$: 1.054327612163653e-01\\
&16&7.7508e-09&9.0495e-08&\\
&32&8.6327e-12&9.7762e-11&\\
&64&4.6582e-14&5.0088e-13&\\
\hline
$\tilde{w}\;,\;\tilde{w}$
&4  &6.6230e-06&1.1888e-03&$L^2$: 1.702510524718458e-03\\
&8  &1.8788e-07&5.8248e-06&$H^1$: 3.514425373878843e-02\\
&16&1.8161e-09&3.1897e-08&\\
&32&2.3060e-12&3.1770e-11&\\
&64&1.1535e-14&1.5150e-13&\\
\hline
$v_j\;,\;\tilde{w}$
&4  &1.7543e-05&0&$L^2$: 8.786063434697107e-3\\
&8  &2.3409e-07&0&$H^1$: 0\\
&16&2.5401e-09&0&\\
&32&3.3059e-12&0&\\
&64&1.4806e-14&0&\\
\hline
$w_j\;,\;\tilde{w}$
&4  &2.3668e-05&0&$L^2$: 1.769711697503764e-03\\
&8  &1.2301e-07&0&$H^1$: 0\\
&16&7.4787e-11&0&\\
&32&4.3801e-14&0&\\
&64&1.9227e-15&0&\\
  \hline
\end{tabular}
\end{table}
As in the previous example, we observe rapid convergence and small
errors.

We now consider integrals involving $v_\alpha\in V^K_m(K)$ satisfying
	\begin{align*}
		-\Delta v_\alpha=x^\alpha\mbox{ in }K\quad,\quad
          v_\alpha=0\mbox{ on }\partial K~, 
	\end{align*}
	with $|\alpha|\leq m$.
	The integrals in \eqref{TargetIntegrals} have the exact values
	\begin{subequations}
	\begin{align}
		\label{InnerProduct1}
		\int_K v_\alpha v_\beta\,dx
		&=4\sum_{k=1}^\infty\sum_{\ell=1}^\infty 
			\frac{S_{\alpha_1,k}S_{\alpha_2,\ell}S_{\beta_1,k}S_{\beta_2,\ell}}{\pi^4(k^2+\ell^2)^2}
		~,\\
		\label{InnerProduct2}
		\int_K\nabla v_\alpha\cdot\nabla w_\alpha~dx
		&=4\sum_{k=1}^\infty\sum_{\ell=1}^\infty 
			\frac{S_{\alpha_1,k}S_{\alpha_2,\ell}S_{\beta_1,k}S_{\beta_2,\ell}}{\pi^2(k^2+\ell^2)}~,
	\end{align}
	\end{subequations}
	where, for integers $a\geq0$ and $\ell\geq1$ we have, following from \cite[Identity (3.761.5)]{Gradshteyn2007}, that
	\begin{align}\label{GradshteynIdentity}
		S_{a,\ell}=\int_0^1 t^a\sin(\ell\pi t)\,dt
		&=(-1)^{\ell+1}\sum_{j=0}^{\floor{a/2}}\frac{(-1)^j}{(\ell\pi)^{2j+1}}\frac{a!}{(a-2j)!}
	          -(-1)^{\floor{a/2}}\frac{a!(a-2\floor{a/2}-1)}{(\ell\pi)^{a+1}}~.
	\end{align}
As before, we obtain reference values that are exact in all digits
shown, up to rounding in the final digit.
Table~\ref{UnitSquareTable} reports the absolute 
        errors of both the $L^2$ inner product as computed
        with~\eqref{L2case}, and the $H^1$ semi-inner product as
        computed with~\eqref{GradientIntegrals} and~\eqref{Grad3}.  Again, we observe rapid convergence and
        small errors.
\begin{table}[]
\centering
\caption{Absolute errors for computing the $L^2$ inner product
  $\int_K v_\alpha v_\beta~dx$ and the $H^1$ semi-inner product
  $\int_K\nabla v_\alpha\cdot\nabla v_\beta~dx$ on the unit square, as
  described in Example \ref{UnitSquareExample}.}
\label{UnitSquareTable}
\begin{tabular}{|cccccc|}\hline
$\alpha$&$\beta$&$n$&$L^2$ error&$H^1$ error&Reference Values
\\\hline
(0,0) 	& (0,0)
& 4 		& 6.6230e-06 	& 1.1888e-03	&$L^2$: 1.702510524718458e-03 \\ 
&& 8 	& 1.8788e-07 	& 5.8248e-06 &$H^1$: 3.514425373878843e-02\\  
&& 16 	& 1.8161e-09 	& 3.1897e-08 &\\ 
&& 32 	& 2.3060e-12 	& 3.1770e-11 &\\ 
&& 64 	& 1.1535e-14 	& 1.5127e-13 &\\ 
\hline
(1,0) 	& (0,0)
& 4 		& 3.1495e-05 	& 5.1747e-04 	&$L^2$: 8.512552623592291e-04 \\ 
&& 8 	& 1.3546e-07 	& 3.0402e-06  &$H^1$: 1.757212686939421e-02\\ 
&& 16 	& 1.2401e-09 	& 1.6264e-08 &\\ 
&& 32 	& 1.5662e-12 	& 1.6000e-11 &\\ 
&& 64 	& 6.4370e-15 	& 1.6175e-14 &\\ 
\hline
(1,1) 	& (1,0)
& 4 		& 2.8944e-05 	& 2.4156e-05  &$L^2$: 2.216128146808729e-04\\ 
&& 8 	& 1.5553e-07 	& 1.0527e-06  &$H^1$: 4.876460403509895e-03\\ 
&& 16 	& 1.2923e-09 	& 4.2780e-09 &\\ 
&& 32 	& 1.6541e-12 	& 2.9498e-12 &\\ 
&& 64 	& 3.6738e-15 	& 7.3119e-14 &\\ 
\hline
(2,1) 	& (0,2)
& 4 		& 1.0205e-05 	& 3.6082e-05 	&$L^2$: 8.101386165180633e-05\\ 
&& 8 	& 7.0511e-08 	& 1.6661e-07 	&$H^1$: 1.905102279276017e-03\\ 
&& 16 	& 6.1937e-10 	& 8.0122e-10 &\\ 
&& 32 	& 7.9987e-13 	& 2.4343e-12 &\\ 
&& 64 	& 7.3959e-15 	& 7.0453e-14 &\\ 
\hline
(4,1) 	& (3,2)
& 4 		& 1.7520e-06 	& 1.5966e-05 	&$L^2$: 9.507439861840766e-06\\ 
&& 8 	& 2.6874e-08 	& 2.4853e-07 	&$H^1$: 3.269201405690909e-04\\ 
&& 16 	& 1.8436e-10 	& 1.1472e-09 &\\ 
&& 32 	& 2.1303e-13 	& 8.4067e-13 &\\ 
&& 64 	& 4.7769e-16 	& 9.2503e-15 &\\ 
\hline
(5,1) 	& (3,3)
& 4 		& 9.5548e-07 	& 1.2468e-05 	&$L^2$: 4.942357655448965e-06 \\ 
&& 8 	& 1.4447e-08 	& 1.1744e-07  &$H^1$: 1.881216015506745e-04\\ 
&& 16 	& 1.0333e-10 	& 4.0048e-10 &\\ 
&& 32 	& 1.2090e-13 	& 1.0942e-13 &\\ 
&& 64 	& 4.3990e-16 	& 3.8299e-17 &\\ 
\hline
(4,2) 	& (4,2)
& 4 		& 1.2419e-06 	& 1.4192e-05 &$L^2$: 4.456767076898193e-06\\ 
&& 8 	& 1.8471e-08 	& 2.1324e-07 	&$H^1$: 1.792263895426231e-04\\ 
&& 16 	& 1.2935e-10 	& 1.0114e-09 &\\ 
&& 32 	& 1.4892e-13 	& 7.2965e-13 &\\ 
&& 64 	& 3.0037e-16 	& 1.1613e-14 &\\ 
\hline
\end{tabular}
\end{table}

\end{example}

\begin{example}[Pac-Man]
  For any constant $\mu>0$, it holds that the function
  $v=r^{\mu}\sin(\mu\theta)$ is harmonic in $\RR^2$, except perhaps at
  the origin. We consider the case where $1/2<\mu<1$, so
  that $v$ has an unbounded gradient at the origin, and
  take $K$ to be the sector of the unit circle given in terms of polar
  coordinates by
  \label{PacManExample}
  \begin{align*}
    K = \{(r,\theta):0< r<1,0<\theta<\pi/\mu\}~.
  \end{align*}
  The boundary $\partial K$ is partitioned into three edges, one
  of which is a circular arc.

  Let $1/2<\nu\leq \mu$, and consider the
  three functions
  \begin{align*}
    v_1=r^{\mu}\sin(\mu\theta)\quad,\quad
    v_2=r^{\nu}\sin(\nu\theta)\quad,\quad
    v_3=(1 - r^2)r^2 \sin(\theta) \sin(\theta - \pi/\mu)~. 
  \end{align*}
  Although $v_1$ and $v_2$ are not in $V_m(K)$ for any $m$, our
  methods for reducing integral involving them to boundary integral
  still apply, e.g.~\eqref{L2caseHarmonic} still holds, and we use them
  below.  Indeed, the construction of $\Phi$ in Proposition~\ref{Laplacephi}
  does not rely on the given Dirichlet data for $v,w$ being in
  $\mathbb{P}_p^*(\partial K)$.  One sees that $v_3\in V_4^{\partial K}(K)$,
as $v=0$ on $\partial K$, and
\begin{align*}
  v_3&=(1-x^2-y^2)(\cos(\pi/\mu)~y-\sin(\pi/\mu)~x)y~,\\
  \Delta v_3&= 2\cos(\pi/\mu)-2\cos(\pi/\mu)\,x^2-14\cos(\pi/\mu)\,y^2+12\sin(\pi/\mu)\,xy~.
\end{align*}  
 We can compute the integrals~\eqref{TargetIntegrals} for these
 functions analytically,
\begin{align*}
  \int_K v_1v_2~dx
  =
  \begin{cases}
    \dfrac{\pi}{4\mu(\mu+1)}~,	&\text{if}~\nu=\mu~,
    \\
    \dfrac{\mu\sin(\nu\pi/\mu)}{(\mu+\nu+2)(\mu^2-\nu^2)}~,
    &\text{if}~\nu<\mu~,
  \end{cases}
      \quad,\quad
      \int_K \nabla v_1\cdot\nabla v_2~dx
      =
      \begin{cases}
        \pi/2~,	&\text{if}~\nu=\mu~,
        \\
        \dfrac{\mu\nu\sin(\nu\pi/\mu)}{\mu^2-\nu^2}~,
        &\text{if}~\nu<\mu~,
      \end{cases}
\end{align*}       
        \begin{align*}
          \int_{K}v_2v_3\,dx=\frac{2\nu\sin(\pi/\mu)\,\sin(\nu\pi/\mu)-4\cos(\pi/\mu)\,(1-\cos(\nu\pi/\mu))}{\nu(\nu+4)(\nu+6)(\nu^2-4)}\quad,\quad
          \int_K \nabla v_2\cdot\nabla v_3\,dx=0~.
        \end{align*}

        In our experiments, we take $\mu=4/7$ and $\nu=2/7$.  The
        absolute errors and (exact) reference values are reported in
        Table~\ref{PacManTable}.  Although we again observe good
        convergence of the quadrature, it is not as rapid as what was
        observed in earlier examples, and the errors do not get near
        machine precision.  This is due the strong singularities of
        the integrands $v_1$ and $v_2$ near the origin.
	\begin{table}[]
		\centering
		\caption{
			Absolute errors for the Pac-Man domain considered in Example \ref{PacManExample},
			using $v_1=r^\mu\sin\mu\theta$ and
                        $v_2=r^\nu\sin\nu\theta$ and $v_3=(1 -
                        r^2)r^2\sin(\theta)\sin(\theta - \pi/\mu)$,
                        with $\mu=4/7$ and $\nu=2/7$.
		}
		\label{PacManTable}
\begin{tabular}{|ccccl|}\hline
Functions &$n$ &$L^2$ error	&$H^1$ error&Reference Values\\\hline
$v_1\;,\;v_1$
&4  &9.4414e-02&2.1575e-01&$L^2$: $49\pi/176$\\
&8  &6.8100e-03&2.1041e-02&$H^1$: $\pi/2$\\
&16&3.0614e-04&9.5614e-04&\\
&32&4.5945e-06&1.4420e-05&\\
&64&2.2640e-08&7.1147e-08&\\
  \hline
$v_1\;,\;v_{2}$  
&4  &8.9563e-02&2.5750e-01&$L^2$: 49/60\\
&8  &8.5462e-03&3.7106e-02&$H^1$: 2/3\\
&16&6.2863e-04&3.5209e-03&\\
&32&1.6028e-05&1.0129e-04&\\
&64&1.6654e-07&5.6503e-07&\\  
  \hline
$v_1\;,\;v_{3}$
&4  &9.0079e-02&0&$L^2$: $16807\sqrt{2}/264960$\\
&8  &4.3964e-04&0&$H^1$: 0\\
&16&1.7055e-05&0&\\
&32&2.5349e-07&0&\\
&64&1.2475e-09&0&\\
\hline
$v_2\;,\;w_3$
&4  &7.7313e-02&0&$L^2$: $2401\sqrt{2}/31680$\\
&8  &6.0152e-04&0&$H^1$: 0\\
&16&5.1225e-05&0&\\
&32&1.4916e-06&0&\\
&64&1.6999e-08&0&\\
\hline
\end{tabular}
\end{table}         
        In practice, the asymptotic behavior of functions in $V_m(K)$
        is known a priori, with the singular behavior near corners of
        $K$ determined only by the angles at these corners
        (cf.~\cite{Wigley1964,Grisvard1985,Grisvard1992}).  Although
        we expect that the level of accuracy achieved by our current
        approach for such integrals is already sufficient for most practical
        computations, it should be possible to further improve the
        accuracy by exploiting the knowledge of the asymptotics near
        the corners.  Singularity subtracting techniques
        (cf.~\cite{Wigley1969,Bruno2009}) have been successfully used
        to this end in related contexts, and we may consider such modifications
        in future work.	
\end{example}

\begin{example}[Puzzle Piece]
	\label{PuzzlePieceExample}
	In previous examples, we have examined the convergence
        rates of our method on relatively simple domains where the
        values of the integrals in \eqref{TargetIntegrals} are
        known exactly, or highly accurate reference values can be
        computed using series expansions.  We now consider a domain $K$, the puzzle piece
        depicted in Figure~\ref{PuzzlePieceFigure}, for which such
        reference values are not available.
        The boundary $\partial K$ is
        partitioned into 12 edges, 4 of which are circular arcs.  The
        cell $K$ can be described in terms of two parameters, the
        radius $r$ of the circular sectors, and the perpendicular distance $b<r$ from
        the centers of these circles to the line containing the two
        adjacent straight edges.  In our case, $r=0.22$ and $b=0.17$.
        The length each of the straight edges is
        $1/2-\sqrt{r^2-b^2}\approx 0.3604$.
        In keeping with common terminology for jigsaw puzzles, we
        will refer to the circular protrusions as ``tabs'' and the
        circular indentions as ``blanks''.  The interior angle at
        four vertices where the tabs meet the straight edges is
        $3\pi/2+\arcsin(b/r)\approx\pi/0.5614$, and the interior angle at the four vertices
        where the blanks meet the straight edges is $\pi/2-\arcsin(b/r)\approx\pi/4.5685$.
	If the tabs are cut off and used to fill the blanks, the
        resulting shape is a unit square, which we used to
        assert that $|K|=1$ in Example~\ref{ConstantFunctionExample}.

        In~\cite{Anand2020}, we provide provide a detailed discussion of how to
        construct a spanning set, and then a basis, for
        $\PP_m^*(\partial K)$ when $K$ has curved edges.  Here, we
        will provide just enough detail to make sense of the
        corresponding numerical experiments.  We begin by labeling the
        vertices $z_0,\ldots, z_{11}$, starting at the lower left
        corner and proceeding counter-clockwise, as seen in Figure~\ref{PuzzlePieceFigure}.
	\begin{figure}[]
		\centering		
		\includegraphics[width=0.45\textwidth]{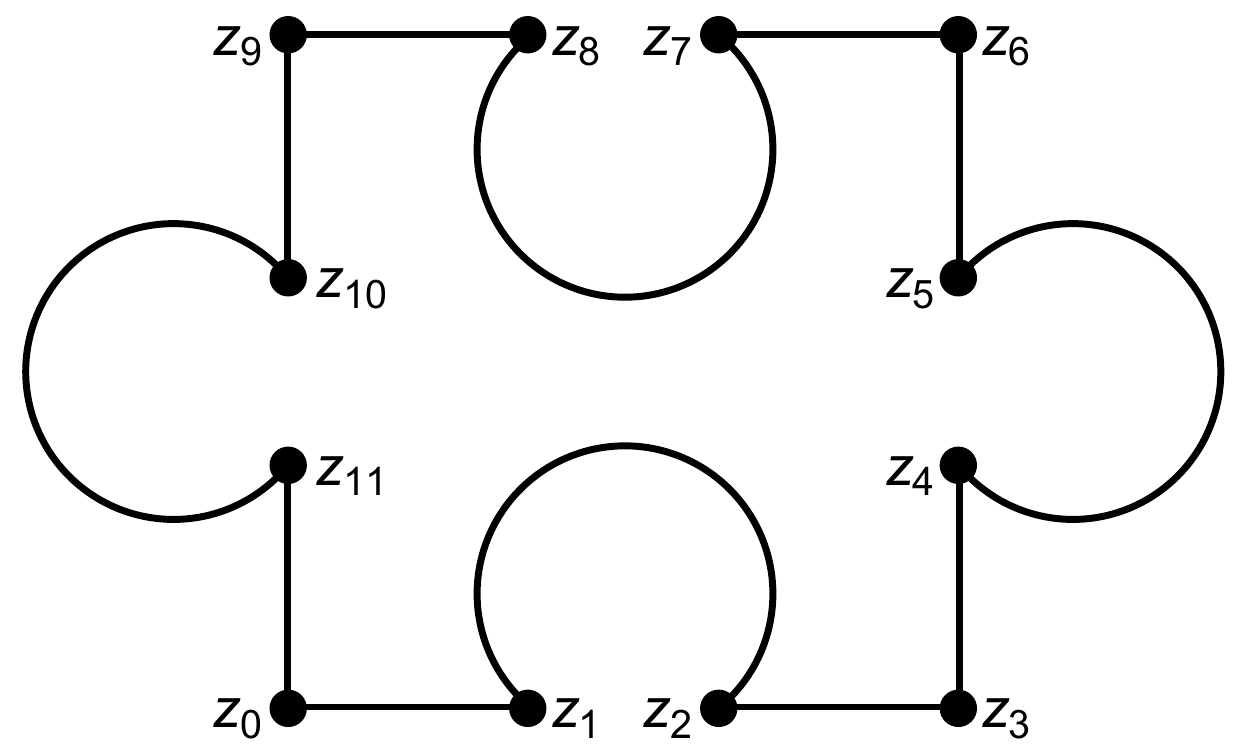}\includegraphics[width=0.45\textwidth]{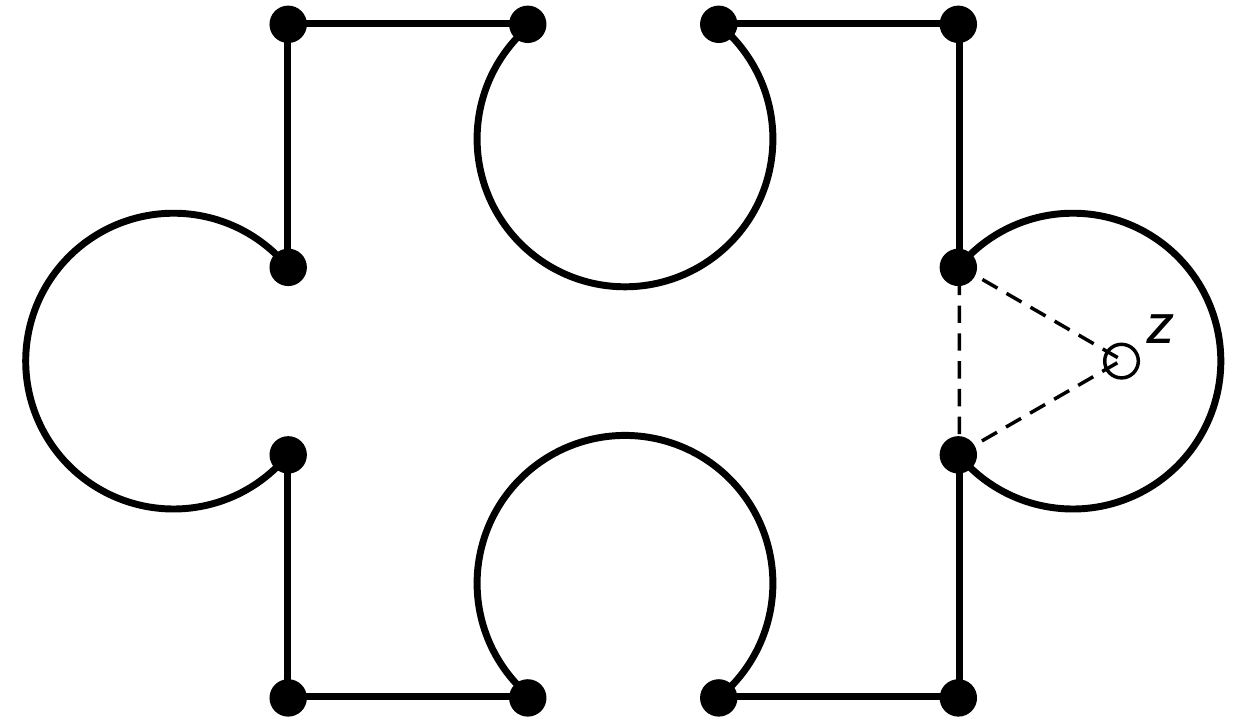}\\[4pt]
                \includegraphics[width=0.30\textwidth]{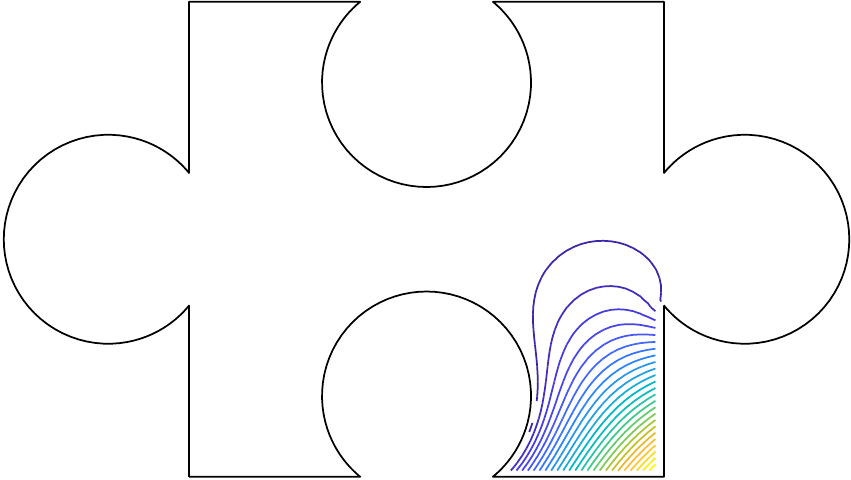}\quad\includegraphics[width=0.30\textwidth]{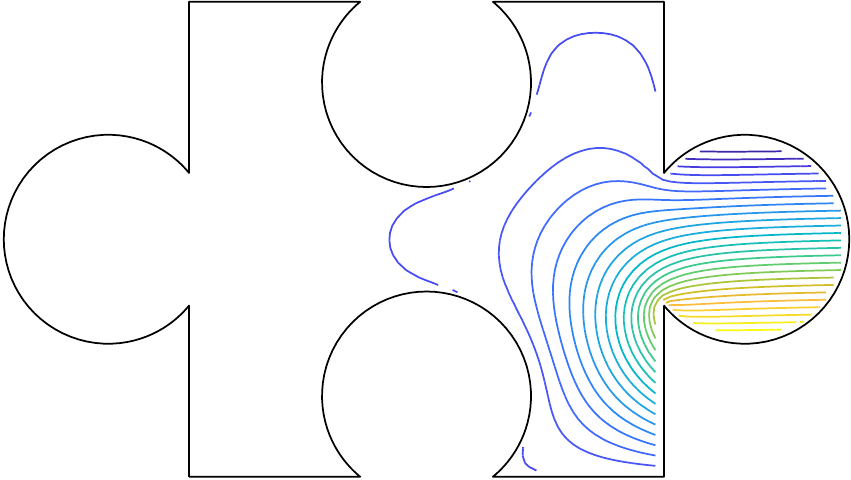}
                \quad\includegraphics[width=0.30\textwidth]{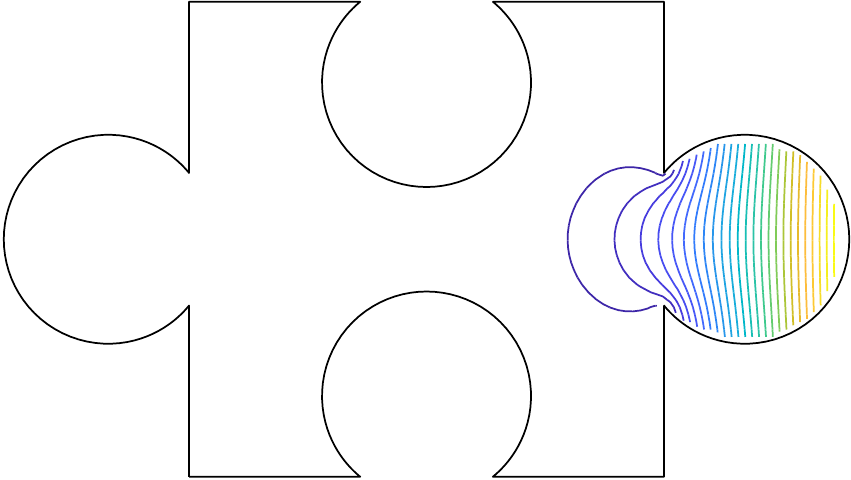}
		\caption{\label{PuzzlePieceFigure} Top Left: The puzzle piece mesh cell considered in
			Example~\ref{PuzzlePieceExample} (and
			Example~\ref{ConstantFunctionExample}).
                        Top Right: The fictitious point $z$ and
                        equilateral triangle used to define $\PP_1(e)$
                        for the circular arc of the right tab; this construction
                        was used to provide boundary values in the
                        for the vertex functions $v_4$ and $v_5$, and
                        the edge function $u_1$.  Bottom Panel:
                        Contour plots of $v_3,v_4,u_1\in
                        V_1^{\partial K}(K)=V_1(K)$, left to right.
                      }
	\end{figure}
      As in Example~\ref{UnitSquareExample}, we let $v_j$ denote a vertex function, 
      i.e. the harmonic function satisfying $v_j(z_i)=\delta_{ij}$ for each of the 12 vertices $z_i$,
      such that, along each edge, the trace of $v_j$ is the trace of a
      linear function.  On straight edges, this corresponds to our
      natural understanding of linear functions of a single variable
      such as arclength, but we must
      clarify what we use for linear functions on curved edges.  Using
      the curved edge $e$ (a circular arc) having $z_4$ and $z_5$ as its endpoints, we
      introduce a fictitious equilateral triangle that includes a
      fictitious point $z$, as seen in
      Figure~\ref{PuzzlePieceFigure}.  The three (linear) barycentric
      coordinates of this triangle are defined in all of $\RR^2$, and
      we take their traces on $e$ to define $\PP_1(e)$.  The traces on
      $e$ of the barycentric coordinates associated with $z_4$ and
      $z_5$ provide part of the Dirichlet data for $v_4$ and $v_5$.
      Analogous constructions are used for the other three curved edges.
      This discussion also makes it clear that, in addition to the 12
      vertex functions in $V_1(K)$, there are four ``edge functions''
      in $V_1(K)$, associated with the four curved edges of $\partial K$, which are harmonic and vanish at every vertex, and
      whose trace on a curved edge is the trace of a linear function
      on that edge, as described above.    We label these edge
      functions $u_0,u_1,u_2,u_3$, where the Dirichlet trace of $u_0$
      is supported on the curved edge of the bottom blank, that of $u_1$ is
      supported on the curved edge of the right tab, that of $u_2$ is
      supported on the curved edge of the top blank, and that of $u_3$ is
      supported on the curved edge of the left tab.  Although we will
      not need to evaluate any of these functions in (the interior of)
      $K$ for our integral computations, we provide contour plots of
      three functions from $V_1(K)$ to further clarify our
      constructions.  We note that $\dim V_1(K)=\dim \PP_1^*(\partial
      K)=16$.  For all curved edges, we choose the equilateral
      triangle so that the fictitious point is directed away from the
      center of $K$.  This has the effect that non-zero Dirichlet
      values of $u_1$ and $u_3$ (associated with the tabs) are
      positive, and the non-zero Dirichlet values of $u_0$ and $u_2$
      (associated with the blanks) are negative; so $u_1,u_3>0$ and
      $u_0,u_2<0$ in the interior of $K$.

      On a straight edge $e$, $\dim\PP_2(e)=3$, but on a circular edge
      $e$, $\dim\PP_2(e)=5$ (cf.~\cite[Proposition 2.1]{Anand2020}).
      From this, it follows that $\dim\PP_2^*(\partial K)=32$.  We
      can augment our basis of $\PP_1^*(\partial K)$ with 16
      additional functions, whose traces on an edge $e$ are traces of 
      quadratic functions $e$, to obtain a basis of $\PP_2^*(\partial K)$.  These, in turn yield functions in
      $V_2^{\partial K}(K)$ by harmonic extension, as before.  We will
      only consider the following 12, which are sufficient for our
      quadrature illustrations: $w_j\in V_2^{\partial K}(K)$ is the
      harmonic function satisfying $w_j=v_jv_{j+1}$ on $\partial K$.
      
      In Table \ref{PuzzlePieceTable}, we use our quadrature method to
      compute the integrals in \eqref{TargetIntegrals} for several
      representative choices of $v,w$ taken from the hierarchical
      basis of $V_2(K)$ we have constructed above.  Apart from the
      known value $\int_K\nabla v\cdot\nabla w\,dx=0$ for
      $v\in V_2^{\partial K}(K)$ and $w\in V_2^{K}(K)$, exact values
      are not available, so we report the computed values to enough
      digits to observe convergence patterns.  Comparing the values
      computed for $n=32$ and $n=64$, we observe that the $L^2$ inner
      products typically differ on the order $10^{-8}$, while
      the $H^1$ semi-inner products typically differ by about
      $10^{-6}$.  These observations are consistent with the absolute
      errors observed in Example \ref{PacManExample}, which similarly
      featured a domain with non-convex corners, giving rise to
      singular functions having unbounded gradients near such corners.
        
\begin{table}[]
		\centering
		\caption{Computed values of the $L^2$ inner product $\int_K vw~dx$ and 
			the $H^1$ semi-inner product $\int_K\nabla v\cdot\nabla w~dx$ on the Puzzle Piece, as described in 
			Example~\ref{PuzzlePieceExample}. Following
                        the notation in that example, $v_j$ denotes
                        the linear vertex functions,
			$u_j$ denotes the linear edge functions,
			$w_j$ denotes the quadratic edge functions, 
			and $\tilde{w}$ denotes the bubble function.}
		\label{PuzzlePieceTable}
		\begin{tabular}{|ccccc|}
			\hline
			$v$		&$w$ 	&$n$ 	&$L^2$ computed	&$H^1$ computed
			\\ \hline
			$v_0$	&$v_0$
			&4 		&1.36722100e-02 	&7.26172648e-01\\
			&&8 	&1.39162733e-02 	&7.24767603e-01\\
			&&16 	&1.39041636e-02 	&7.25599803e-01\\
			&&32 	&1.39043324e-02 	&7.25576824e-01\\
			&&64 	&1.39043346e-02 	&7.25576695e-01\\
			\hline
			$v_0$	&$v_1$
			&4 		&1.05818498e-02 	&-5.48518347e-01\\
			&&8 	&9.16414178e-03 	&-5.66262962e-01\\
			&&16 	&9.17500217e-03 	&-5.66136764e-01\\
			&&32 	&9.17618664e-03 	&-5.66201634e-01\\
			&&64 	&9.17618833e-03 	&-5.66201663e-01\\
			\hline
			$v_0$	&$w_0$
			&4 		&1.81794593e-03 	&1.28954443e-01\\
			&&8 	&2.01809269e-03 	&1.24305855e-01\\
			&&16 	&2.01072985e-03 	&1.24548106e-01\\
			&&32 	&2.01040843e-03 	&1.24569497e-01\\
			&&64 	&2.01040886e-03 	&1.24569472e-01\\
			\hline 
			$v_1$	&$u_0$
			&4 		&-8.30808861e-03 	&-8.59205540e-01\\
			&&8 	&-1.08862664e-02 	&-1.10211908e+00\\
			&&16 	&-1.07049968e-02 	&-1.09593380e+00\\
			&&32 	&-1.07051888e-02 	&-1.09590692e+00\\
			&&64 	&-1.07051900e-02 	&-1.09590691e+00\\
			\hline
			$u_0$	&$u_0$
			&4 		&1.00392937e-01 	&6.09807959e+00\\
			&&8 	&1.27777179e-01 	&7.38466056e+00\\
			&&16 	&1.27439875e-01 	&7.37314091e+00\\
			&&32 	&1.27460415e-01 	&7.37307150e+00\\
			&&64 	&1.27460423e-01 	&7.37307096e+00\\
			\hline
			$u_0$	&$u_4$
			&4 		&-5.39752362e-03 	&1.30869343e-01\\
			&&8 	&-3.92970591e-03 	&9.61489506e-02\\
			&&16 	&-3.91163646e-03 	&9.47669194e-02\\
			&&32 	&-3.92266750e-03 	&9.50284316e-02\\
			&&64 	&-3.92268446e-03 	&9.50288434e-02\\
			\hline 
			$\tilde{w}	$	&$\tilde{w}$
			&4 		&1.98651110e-04 	&1.19301041e-02\\
			&&8 	&1.36096130e-04 	&9.87645211e-03\\
			&&16 	&1.36275143e-04 	&9.85192757e-03\\
			&&32 	&1.36415508e-04 	&9.85631214e-03\\
			&&64 	&1.36415772e-04 	&9.85632205e-03\\
			\hline
			$v_0$	&$\tilde{w}$
			&4 		&4.06902400e-04 	&0\\
			&&8 	&2.35312339e-04 	&0\\
			&&16 	&2.35317337e-04 	&0\\
			&&32 	&2.35506723e-04 	&0\\
			&&64 	&2.35507154e-04 	&0\\
			\hline 
			$u_3$	&$\tilde{w}$
			&4 		&-5.49240260e-04 	&0\\
			&&8 	&-1.07144403e-03 	&0\\
			&&16 	&-1.06639357e-03 	&0\\
			&&32 	&-1.06754253e-03 	&0\\
			&&64 	&-1.06754457e-03 	&0\\
			\hline 
		\end{tabular}
	\end{table}

\end{example}

%%%%%%%%%%%%%%%%%%%%%%%%%%%%%%%%%%%%%%%%%%%%%%%%%%%%%%%%%%%%%%%%%
%%%%%%%%%%%%%%%%%%%%%%%%%%%%%%%%%%%%%%%%%%%%%%%%%%%%%%%%%%%%%%%%%

\section{Additional Details}\label{Details}

In~\cite{Ovall2018} we discuss, in detail, integral equation
techniques to compute interior function values and derivatives, as
well as the normal derivative, of harmonic functions with prescribed
Dirichlet data.  That paper also involves a detailed discussion of
the so-called \textit{Kress quadrature} that underlies much of the
practical computations.  Readers interested in that level of detail
for these aspects of our present work are referred to that paper.
Here, we merely outline key components to make this paper a bit
more self-contained, and to explain the parameters $n$ and $\sigma$
used for the experiments in the Section~\ref{Experiments}.

A fundamental step in the techniques in~\cite{Ovall2018} is the
computation of a harmonic conjugate $\hat\phi$ of a harmonic
function $\phi$ that is given implicitly by its Dirichlet data.  The associated
boundary value problem for $\hat\phi$ is given in~\eqref{HarmConjNeum}, and we
emphasize its general form as a Neumann problem,
\begin{align*}
\Delta\hat\phi=0\mbox{ in
  }K\quad,\quad\frac{\partial\hat\phi}{\partial\mb{n}}=g\mbox{ on
  }\partial K\quad,\quad \int_{\partial K}\hat\phi\,ds=0~.  
\end{align*}
as such problems also arise in the construction described in
Proposition~\ref{Laplacephi}.  Here, the Neumann data $g$ is specified
by the user.  We provide the associated integral equation employed in
our work, whose
solution yields a Neumann-to-Dirichlet map $g\to
\hat\phi\vert_{\partial K}$.  We only present the equation in the case
that $\partial K$ has a single corner $z$.  The case of multiple
corners is dealt with analogously, and further details are provided
in~\cite{Ovall2018}.  The associated integral equation is
\begin{align}\label{IntegralEquation}
  |\partial
  K|\hat\phi(z)+\frac{\hat\phi(x)-\hat\phi(z)}{2}+\int_{\partial
  K}\left(\frac{\partial
  G(x,y)}{\partial\mb{n}(y)}+1\right)(\hat\phi(y)-\hat\phi(z))\,ds(y)=\int_{\partial
  K} G(x,y)g(y)\,ds(y)~,
\end{align}
where $G(x,y)=-(2\pi)^{-1}\ln|x-y|$ is the fundamental solution for
the Laplacian in 2D.  This integral equation is discretized by a
Nystr\"om method, which is a quadrature based method in which the
integrals from~\eqref{IntegralEquation} in the variable $y$ are replaced
by a quadrature that is suitable for these integrands regardless of the
choice of $x\in\partial\Omega$, resulting in a ``semi-discrete''
equation.  The semi-discrete equation is then sampled at the
quadrature points to obtain a fully discrete (and well-conditioned) square
linear system.  We solve this system using GMRES.
In the case that $\hat\phi$ is a harmonic conjugate of the harmonic
function $\phi$ having Dirichlet data $f$, a Dirichlet-to-Neumann map
$f\mapsto\partial\phi/\partial\mb{n}$ is obtained by taking the
tangential derivative of $\hat\phi$,
$\partial\phi/\partial\mb{n}=\partial\hat\phi/\partial\mb{t}$.
In our computations, an FFT is used to efficiently approximate
$\partial\hat\phi/\partial\mb{t}$ at quadrature points from the
approximation of $\hat\phi$ at these points that was obtained by
solving the Nystr\"om linear system.

The fundamental quadrature used in our discretization is due to
Kress~\cite{Kress1990}, and we briefly describe it here to explain the
parameters $n$ and $\sigma$ used for the experiments.  Suppose that
a function $F$ is continuous on $[a,b]$.   For any suitable
change-of-variable $t=\lambda(\tau)$, $\lambda:[a,b]\to[a,b]$, we have
$\int_a^bF(t)\,dt=\int_a^bF(\lambda(\tau))\lambda'(\tau)\,d\tau$.
A clever choice of $\lambda$ ensures that the new integrand,
$F(\lambda(\tau))\lambda'(\tau)$, vanishes at the endpoints $a$ and
$b$, together with some of its derivatives.  It is well-known that 
the trapezoid rule converges rapidly for such integrands,
with the rate of convergence depending on the order at which
$F(\lambda(\tau))\lambda'(\tau)$ vanishes at these endpoints.
\textit{Kress quadrature} is obtained by using the trapezoid rule
after the following sigmoidal change-of-variable, which is determined
by single integer parameter $\sigma\geq 2$,
\begin{align}
  \label{KressCOVa}
  \lambda(\tau)=(b-a)\frac{[c(\tau)]^\sigma}{[c(\tau)]^\sigma+[1-c(\tau)]^\sigma}+a\quad,\quad
  c(\tau)=\left(\frac{1}{2}-\frac{1}{\sigma}\right)\theta^2+\frac{\theta}{\sigma}+\frac{1}{2}~,
\end{align}
where $\theta=(2\tau-a-b)/(b-a)$.  A key property of this
change-of-variable is that $\lambda'(\tau)$ has roots of order
$\sigma-1$ at $\tau=a$ and $\tau=b$.  Letting $h=(b-a)/m$ and
$\tau_k=a+kh$, for $0\leq k\leq m$, the corresponding quadrature is
\begin{align}\label{KressQuad}
  \int_a^bF(t)\,dt\approx\sideset{}{''}\sum_{k=0}^mF(t_k)w_k=\sum_{k=1}^{m-1}F(t_k)w_k
  \quad,\quad t_k=\lambda(\tau_k)\quad,\quad w_k=h\lambda'(\tau_k)~.
\end{align}
The double-prime notation in the first sum above indicates that its
initial and final terms are halved.  Since $w_0=w_{m}=0$, these terms
can be dropped.  Kress provides an analysis of this quadrature based
on the Euler-Maclaurin formula in~\cite{Kress1990}, and we provide a
complementary analysis based on Fourier series in~\cite{Ovall2018}.
Depending on whether both, one or neither of the endpoints are
counted,~\eqref{KressQuad} involves between $m-1$ and $m+1$ quadrature
points.  It is convenient in our applications to consider it as an
$m$-point quadrature by including only one of the endpoints, so that
$m$ points are ``assigned'' to each edge of $\partial K$ via a
parameterization of the boundary, and each vertex of $\partial K$ is
counted only once.  It is also convenient in our implementation to
take $m=2n$.  This choice applies both to the quadrature used for the Nystr\"om
linear system, which yields our discrete Dirichlet-to-Neumann map, and
for the final boundary quadratures that employ this data to compute
the target integrals~\eqref{TargetIntegrals} as described in the
introductory paragraph of Section~\ref{Experiments}.  When the integer
parameter $n$ is used in Section~\ref{Experiments}, this is what it
indicates.

\iffalse
\begin{itemize}
  \item Discussion of Kress quadrature, along the lines given in
    ~\cite{Ovall2018}
  \item Discussion of algorithms for polynomial products and gradients
    (normal derivatives).  Normal derivative only needed in~\eqref{L2case}.
  \item Try to find clever ways to reduce $\int_K(x-z)^\alpha vw\,dx$
   to the boundary, where $v,w\in V_m(K)$ are harmonic.
  \item Try to find clever ways to reduce $\int_K(x-z)^\alpha \nabla
    v\cdot\nabla w\,dx$
   to the boundary, where $v,w\in V_m(K)$ are harmonic.
  \end{itemize}
\fi
  
We include an extension of Table~\ref{Palpha} as a ready reference
that may be used for higher degree polynomials.  This is given in
Table~\ref{Palpha2} and covers multiindices $\alpha$ for which
$7\leq|\alpha|\leq 10$.  We also make a brief note on computing the
products of polynomials in our context.  Given
$p,q\in\mathbb{P}_m(\mathbb{R}^2)$, with
\[
	p(x)=\sum_{|\alpha|\leq m}c_\alpha(x-z)^\alpha~,
	\quad
	q(x)=\sum_{|\beta|\leq m}d_\beta(x-z)^\beta~,
\]
the product $pq\in\mathbb{P}_{2m}(\mathbb{R}^2)$ may be written as
\[
	(pq)(x)=\sum_{|\gamma|\leq 2m}\tilde{c}_\gamma(x-z)^\gamma~,
\]
where the coefficients are given by
\[
	\tilde{c}_\gamma=\sum_{|\alpha|+|\beta|=|\gamma|}c_\alpha d_\beta~.
\]
In the case where $p,q$ are sparse, i.e. the majority of their
coefficients are zero, as is the case for our experiments, we execute
a double loop over the nonzero coefficients of $p$ and $q$, and for
each pair $(\alpha,\beta)$, we determine the integer $n$ corresponding
to the multi-index $\alpha+\beta$, according to the bijection
described in \eqref{MultiIndexMaps}.  We then increment the
coefficient $\tilde{c}_\gamma$, with $\gamma$ corresponding to $n$,
with $\tilde{c}_\gamma\leftarrow\tilde{c}_\gamma+c_\alpha d_\beta$.
\begin{table}
  \centering
  \caption{\label{Palpha2} An enumeration of multiindices for
    $7\leq|\alpha|\leq 10$, together with the indices and values of the
    non-zero coefficients of $P_\alpha$ from Proposition~\ref{Laplacep}.}
\begin{tabular}{|cccc|}\hline
  $k$&$\alpha$&$P_\alpha$ indices&$P_\alpha$ coefficients\\\hline
 28&$(7,0)$&$(45, 47, 49, 51, 53)$&$(85, 12, -42, 28, -3)/6144$\\
 29&$(6,1)$&$(46, 48, 50, 52, 54)$&$(247, 84, -126, 36, -1)/14336$\\
 30&$(5,2)$&$(45, 47, 49, 51, 53)$&$(-73, 2628, 1554, -1036, 111)/129024$\\
 31&$(4,3)$&$(46, 48, 50, 52, 54)$&$(-489, 4564, 3906, -1116, 31)/215040$\\
 32&$(3,4)$&$(45, 47, 49, 51, 53)$&$(31, -1116, 3906, 4564, -489)/215040$\\
 33&$(2,5)$&$(46, 48, 50, 52, 54)$&$(111, -1036, 1554, 2628, -73)/129024$\\
 34&$(1,6)$&$(45, 47, 49, 51, 53)$&$(-1, 36, -126, 84, 247)/14336$\\
 35&$(0,7)$&$(46, 48, 50, 52, 54)$&$(-3, 28, -42, 12, 85)/6144$\\\hline
 36&$(8,0)$&$(55, 57, 59, 61, 63, 65)$&$(511, 45, -210, 210, -45, 1)/46080$\\
 37&$(7,1)$&$(56, 58, 60, 62, 64)$&$(251, 60, -126, 60, -5)/18432$\\
 38&$(6,2)$&$(55, 57, 59, 61, 63, 65)$&$(-233, 10485, 4830, -4830, 1035, -23)/645120$\\
 39&$(5,3)$&$(56, 58, 60, 62, 64)$&$(-191, 2292, 1638, -780, 65)/129024$\\
 40&$(4,4)$&$(55, 57, 59, 61, 63, 65)$&$(1, -45, 210, 210, -45, 1)/12600$\\
 41&$(3,5)$&$(56, 58, 60, 62, 64)$&$(65, -780, 1638, 2292, -191)/129024$\\
 42&$(2,6)$&$(55, 57, 59, 61, 63, 65)$&$(-23, 1035, -4830, 4830, 10485, -233)/645120$\\
 43&$(1,7)$&$(56, 58, 60, 62, 64)$&$(-5, 60, -126, 60, 251)/18432$\\
 44&$(0,8)$&$(55, 57, 59, 61, 63, 65)$&$(1, -45, 210, -210, 45, 511)/46080$\\\hline
 45&$(9,0)$&$(66, 68, 70, 72, 74, 76)$&$(93, 5, -30, 42, -15, 1)/10240$\\
 46&$(8,1)$&$(67, 69, 71, 73, 75, 77)$&$(1013, 165, -462, 330, -55, 1)/92160$\\
 47&$(7,2)$&$(66, 68, 70, 72, 74, 76)$&$(-11, 605, 210, -294, 105, -7)/46080$\\
 48&$(6,3)$&$(67, 69, 71, 73, 75, 77)$&$(-53, 795, 462, -330, 55, -1)/53760$\\
 49&$(5,4)$&$(66, 68, 70, 72, 74, 76)$&$(29, -1595, 9570, 8106, -2895, 193)/645120$\\
 50&$(4,5)$&$(67, 69, 71, 73, 75, 77)$&$(193, -2895, 8106, 9570, -1595, 29)/645120$\\
 51&$(3,6)$&$(66, 68, 70, 72, 74, 76)$&$(-1, 55, -330, 462, 795, -53)/53760$\\
 52&$(2,7)$&$(67, 69, 71, 73, 75, 77)$&$(-7, 105, -294, 210, 605, -11)/46080$\\
 53&$(1,8)$&$(66, 68, 70, 72, 74, 76)$&$(1, -55, 330, -462, 165, 1013)/92160$\\
 54&$(0,9)$&$(67, 69, 71, 73, 75, 77)$&$(1, -15, 42, -30, 5,93)/10240$\\\hline
  55&$(10,0)$&$(78, 80, 82, 84, 86, 88, 90)$&$(2047, 66, -495, 924, -495, 66, -1)/270336$\\
  56&$(9,1)$&$(79, 81, 83, 85, 87, 89)$&$(509, 55, -198, 198, -55, 3)/56320$\\
  57&$(8,2)$&$(78, 80, 82, 84, 86, 88, 90)$&$(-1981, 130746, 33165, -61908, 33165, -4422, 67)/12165120$\\
  58&$(7,3)$&$(79, 81, 83, 85, 87, 89)$&$(-681, 12485, 5742, -5742, 1595, -87)/1013760$\\
  59&$(6,4)$&$(78, 80, 82, 84, 86, 88, 90)$&$(743, -49038, 367785, 259644, -139095, 18546, -281)/28385280$\\
  60&$(5,5)$&$(79, 81, 83, 85, 87, 89)$&$(3, -55, 198, 198, -55, 3)/16632$\\
  61&$(4,6)$&$(78, 80, 82, 84, 86, 88, 90)$&$(-281, 18546, -139095, 259644, 367785, -49038, 743)/28385280$\\
  62&$(3,7)$&$(79, 81, 83, 85, 87, 89)$&$(-87, 1595, -5742, 5742, 12485, -681)/1013760$\\
  63&$(2,8)$&$(78, 80, 82, 84, 86, 88, 90)$&$(67, -4422, 33165, -61908, 33165, 130746, -1981)/12165120$\\
  64&$(1,9)$&$(79, 81, 83, 85, 87, 89)$&$(3, -55, 198, -198, 55, 509)/56320$\\
  65&$(0,10)$&$(78, 80, 82, 84, 86, 88, 90)$&$(-1, 66, -495, 924, -495, 66, 2047)/270336$\\ \hline
\end{tabular}
\end{table}

All of the numerical experiments in this paper were implemented in \textsc{Matlab}.
Those interested in obtaining a copy of our code may do so at the GitHub repository found at
\[
\texttt{https://github.com/samreynoldsmath/HigherOrderCurvedElementQuadrature}
\]

%%%%%%%%%%%%%%%%%%%%%%%%%%%%%%%%%%%%%%%%%%%%%%%%%%%%%%%%%%%%%%%%%%%%%%%%%%%%
%%%%%%%%%%%%%%%%%%%%%%%%%%%%%%%%%%%%%%%%%%%%%%%%%%%%%%%%%%%%%%%%%%%%%%%%%%%%
\section{Concluding Remarks}

We have described an approach for integrating products of functions
and their gradients over curvilinear polygons, when the functions are
given implicitly in terms of Poisson problems involving polynomial
data.  The efficient and accurate evaluation of such integrals is
instrumental in the formation of linear systems associated with finite
element methods such as Trefftz-FEM and VEM, which employ non-standard
meshes.  In our approach, integrals on curvilinear polygonal cells are
reduced to integrals along their boundaries, using constructions that
yield a function whose Laplacian is a given polynomial or harmonic
function, together with integration by parts.  The data for quadrature
approximations of these boundary integrals is obtained using
Dirichlet-to-Neumann and Neumann-to-Dirichlet maps that are computed
using techniques based on second-kind integral equations.  Numerical
examples demonstrate rapid convergence and high accuracy of these
quadratures.

%%%%%%%%%%%%%%%%%%%%%%%%%%%%%%%%%%%%%%%%%%%%%%%%%%%%%%%%%%%%%%%%%%%%%%%%%%%%
%%%%%%%%%%%%%%%%%%%%%%%%%%%%%%%%%%%%%%%%%%%%%%%%%%%%%%%%%%%%%%%%%%%%%%%%%%%%
%\section{Scraps}

%%%%%%%%%%%%%%%%
% begin bibliography
%%%%%%%%%%%%%%%%
\def\cprime{$'$}

%%%%%%%%%%%%%%%%
% end bibliography
%%%%%%%%%%%%%%%%


\begin{thebibliography}{10}

\bibitem{Ahmad2013}
B.~Ahmad, A.~Alsaedi, F.~Brezzi, L.~D. Marini, and A.~Russo.
\newblock Equivalent projectors for virtual element methods.
\newblock {\em Comput. Math. Appl.}, 66(3):376--391, 2013.

\bibitem{Anand2020}
A.~Anand, J.~S. Ovall, S.~E. Reynolds, and S.~Wei{\ss}er.
\newblock Trefftz {F}inite {E}lements on {C}urvilinear {P}olygons.
\newblock {\em SIAM J. Sci. Comput.}, 42(2):A1289--A1316, 2020.

\bibitem{Anand2018}
A.~Anand, J.~S. Ovall, and S.~Wei{\ss}er.
\newblock A {N}ystr\"{o}m-based finite element method on polygonal elements.
\newblock {\em Comput. Math. Appl.}, 75(11):3971--3986, 2018.

\bibitem{Antonietti2014}
P.~F. Antonietti, L.~Beir{\~a}o~da Veiga, D.~Mora, and M.~Verani.
\newblock A stream virtual element formulation of the {S}tokes problem on
  polygonal meshes.
\newblock {\em SIAM J. Numer. Anal.}, 52(1):386--404, 2014.

\bibitem{Antonietti2019}
P.~F. Antonietti, S.~Berrone, M.~Verani, and S.~Wei{\ss}er.
\newblock The virtual element method on anisotropic polygonal discretizations.
\newblock In {\em Numerical mathematics and advanced applications---{ENUMATH}
  2017}, volume 126 of {\em Lect. Notes Comput. Sci. Eng.}, pages 725--733.
  Springer, Cham, 2019.

\bibitem{Antonietti2017}
P.~F. Antonietti, M.~Bruggi, S.~Scacchi, and M.~Verani.
\newblock On the virtual element method for topology optimization on polygonal
  meshes: a numerical study.
\newblock {\em Comput. Math. Appl.}, 74(5):1091--1109, 2017.

\bibitem{Antonietti2018}
P.~F. Antonietti, P.~Houston, and G.~Pennesi.
\newblock Fast numerical integration on polytopic meshes with applications to
  discontinuous {G}alerkin finite element methods.
\newblock {\em J. Sci. Comput.}, 77(3):1339--1370, 2018.

\bibitem{Artioli2020}
E.~Artioli, A.~Sommariva, and M.~Vianello.
\newblock Algebraic cubature on polygonal elements with a circular edge.
\newblock {\em Comput. Math. Appl.}, 79(7):2057--2066, 2020.

\bibitem{BeiraodaVeiga2020}
L.~Beir\~{a}o~da Veiga, F.~Brezzi, L.~D. Marini, and A.~Russo.
\newblock Polynomial preserving virtual elements with curved edges.
\newblock {\em Math. Models Methods Appl. Sci.}, 30(8):1555--1590, 2020.

\bibitem{BeiraodaVeiga2019}
L.~Beir\~{a}o~da Veiga, A.~Russo, and G.~Vacca.
\newblock The virtual element method with curved edges.
\newblock {\em ESAIM Math. Model. Numer. Anal.}, 53(2):375--404, 2019.

\bibitem{BeiraodaVeiga2013}
L.~Beir{\~a}o~da Veiga, F.~Brezzi, A.~Cangiani, G.~Manzini, L.~D. Marini, and
  A.~Russo.
\newblock Basic principles of virtual element methods.
\newblock {\em Math. Models Methods Appl. Sci.}, 23(1):199--214, 2013.

\bibitem{BeiraodaVeiga2014}
L.~Beir{\~a}o~da Veiga, F.~Brezzi, L.~D. Marini, and A.~Russo.
\newblock The hitchhiker's guide to the virtual element method.
\newblock {\em Math. Models Methods Appl. Sci.}, 24(8):1541--1573, 2014.

\bibitem{BeiraodaVeiga2016}
L.~Beir{\~a}o~da Veiga, A.~Chernov, L.~Mascotto, and A.~Russo.
\newblock Basic principles of {$hp$} virtual elements on quasiuniform meshes.
\newblock {\em Math. Models Methods Appl. Sci.}, 26(8):1567--1598, 2016.

\bibitem{BeiraodaVeiga2014a}
L.~Beir{\~a}o~da Veiga and G.~Manzini.
\newblock A virtual element method with arbitrary regularity.
\newblock {\em IMA J. Numer. Anal.}, 34(2):759--781, 2014.

\bibitem{Benedetto2016}
M.~F. Benedetto, S.~Berrone, A.~Borio, S.~Pieraccini, and S.~Scial\`o.
\newblock Order preserving {SUPG} stabilization for the virtual element
  formulation of advection-diffusion problems.
\newblock {\em Comput. Methods Appl. Mech. Engrg.}, 311:18--40, 2016.

\bibitem{Brezzi2014}
F.~Brezzi and L.~D. Marini.
\newblock Virtual element and discontinuous {G}alerkin methods.
\newblock In {\em Recent developments in discontinuous {G}alerkin finite
  element methods for partial differential equations}, volume 157 of {\em IMA
  Vol. Math. Appl.}, pages 209--221. Springer, Cham, 2014.

\bibitem{Bruno2009}
O.~P. Bruno, J.~S. Ovall, and C.~Turc.
\newblock A high-order integral algorithm for highly singular {PDE} solutions
  in {L}ipschitz domains.
\newblock {\em Computing}, 84(3-4):149--181, 2009.

\bibitem{Copeland2009}
D.~Copeland, U.~Langer, and D.~Pusch.
\newblock From the boundary element domain decomposition methods to local
  {T}refftz finite element methods on polyhedral meshes.
\newblock In {\em Domain decomposition methods in science and engineering
  {XVIII}}, volume~70 of {\em Lect. Notes Comput. Sci. Eng.}, pages 315--322.
  Springer, Berlin, 2009.

\bibitem{Floater2015}
M.~S. Floater.
\newblock Generalized barycentric coordinates and applications.
\newblock {\em Acta Numer.}, 24:161--214, 2015.

\bibitem{Gain2014}
A.~L. Gain, C.~Talischi, and G.~H. Paulino.
\newblock On the {V}irtual {E}lement {M}ethod for three-dimensional linear
  elasticity problems on arbitrary polyhedral meshes.
\newblock {\em Comput. Methods Appl. Mech. Engrg.}, 282:132--160, 2014.

\bibitem{Gradshteyn2007}
I.~S. Gradshteyn and I.~M. Ryzhik.
\newblock {\em Table of integrals, series, and products}.
\newblock Elsevier/Academic Press, Amsterdam, seventh edition, 2007.
\newblock Translated from the Russian, Translation edited and with a preface by
  Alan Jeffrey and Daniel Zwillinger, With one CD-ROM (Windows, Macintosh and
  UNIX).

\bibitem{Grisvard1985}
P.~Grisvard.
\newblock {\em Elliptic problems in nonsmooth domains}, volume~24 of {\em
  Monographs and Studies in Mathematics}.
\newblock Pitman (Advanced Publishing Program), Boston, MA, 1985.

\bibitem{Grisvard1992}
P.~Grisvard.
\newblock {\em Singularities in boundary value problems}, volume~22 of {\em
  Recherches en Math\'ematiques Appliqu\'ees [Research in Applied
  Mathematics]}.
\newblock Masson, Paris, 1992.

\bibitem{Hofreither2011}
C.~Hofreither.
\newblock {$L_2$} error estimates for a nonstandard finite element method on
  polyhedral meshes.
\newblock {\em J. Numer. Math.}, 19(1):27--39, 2011.

\bibitem{Hofreither2010}
C.~Hofreither, U.~Langer, and C.~Pechstein.
\newblock Analysis of a non-standard finite element method based on boundary
  integral operators.
\newblock {\em Electron. Trans. Numer. Anal.}, 37:413--436, 2010.

\bibitem{Hofreither2016}
C.~Hofreither, U.~Langer, and S.~Wei{\ss}er.
\newblock Convection-adapted {BEM}-based {FEM}.
\newblock {\em ZAMM Z. Angew. Math. Mech.}, 96(12):1467--1481, 2016.

\bibitem{Hormann2018}
K.~Hormann and N.~Sukumar, editors.
\newblock {\em Generalized barycentric coordinates in computer graphics and
  computational mechanics}.
\newblock CRC Press, Boca Raton, FL, 2018.

\bibitem{Karachik2010}
V.~V. Karachik and N.~A. Antropova.
\newblock On the solution of a nonhomogeneous polyharmonic equation and the
  nonhomogeneous {H}elmholtz equation.
\newblock {\em Differ. Uravn.}, 46(3):384--395, 2010.

\bibitem{Kress1990}
R.~Kress.
\newblock A {N}ystr\"om method for boundary integral equations in domains with
  corners.
\newblock {\em Numer. Math.}, 58(2):145--161, 1990.

\bibitem{Natarajan2009}
S.~Natarajan, S.~Bordas, and D.~R. Mahapatra.
\newblock Numerical integration over arbitrary polygonal domains based on
  {S}chwarz-{C}hristoffel conformal mapping.
\newblock {\em Internat. J. Numer. Methods Engrg.}, 80(1):103--134, 2009.

\bibitem{Ovall2018}
J.~S. Ovall and S.~E. Reynolds.
\newblock A high-order method for evaluating derivatives of harmonic functions
  in planar domains.
\newblock {\em SIAM J. Sci. Comput.}, 40(3):A1915--A1935, 2018.

\bibitem{Rjasanow2012}
S.~Rjasanow and S.~Wei{\ss}er.
\newblock Higher order {BEM}-based {FEM} on polygonal meshes.
\newblock {\em SIAM J. Numer. Anal.}, 50(5):2357--2378, 2012.

\bibitem{Rjasanow2014}
S.~Rjasanow and S.~Wei{\ss}er.
\newblock F{EM} with {T}refftz trial functions on polyhedral elements.
\newblock {\em J. Comput. Appl. Math.}, 263:202--217, 2014.

\bibitem{Seibel2020}
D.~Seibel and S.~Wei{\ss}er.
\newblock Recovery-based error estimators for the {VEM} and {BEM}-based {FEM}.
\newblock {\em Comput. Math. Appl.}, 80(9):2073--2089, 2020.

\bibitem{Sommariva2007}
A.~Sommariva and M.~Vianello.
\newblock Product {G}auss cubature over polygons based on {G}reen's integration
  formula.
\newblock {\em BIT}, 47(2):441--453, 2007.

\bibitem{Sudhakar2014}
Y.~Sudhakar, J.~M. de~Almeida, and W.~A. Wall.
\newblock An accurate, robust, and easy-to-implement method for integration
  over arbitrary polyhedra: Application to embedded interface methods.
\newblock {\em J. Comput. Phys.}, 273:393 -- 415, 2014.

\bibitem{Sudhakar2017}
Y.~Sudhakar, A.~Sommariva, M.~Vianello, and W.~A. Wall.
\newblock On the use of compressed polyhedral quadrature formulas in embedded
  interface methods.
\newblock {\em SIAM J. Sci. Comput.}, 39(3):B571--B587, 2017.

\bibitem{Talischi2014}
C.~Talischi and G.~H. Paulino.
\newblock Addressing integration error for polygonal finite elements through
  polynomial projections: a patch test connection.
\newblock {\em Math. Models Methods Appl. Sci.}, 24(8):1701--1727, 2014.

\bibitem{Weisser2011}
S.~Wei{\ss}er.
\newblock Residual error estimate for bem-based fem on polygonal meshes.
\newblock {\em Numer. Math.}, 118:765--788, 2011.
\newblock 10.1007/s00211-011-0371-6.

\bibitem{Weisser2014}
S.~Wei{\ss}er.
\newblock Arbitrary order {T}refftz-like basis functions on polygonal meshes
  and realization in {BEM}-based {FEM}.
\newblock {\em Comput. Math. Appl.}, 67(7):1390--1406, 2014.

\bibitem{Weisser2017}
S.~Wei{\ss}er.
\newblock Residual based error estimate and quasi-interpolation on polygonal
  meshes for high order {BEM}-based {FEM}.
\newblock {\em Comput. Math. Appl.}, 73(2):187--202, 2017.

\bibitem{Weisser2019a}
S.~Wei{\ss}er.
\newblock Anisotropic polygonal and polyhedral discretizations in finite
  element analysis.
\newblock {\em ESAIM Math. Model. Numer. Anal.}, 53(2):475--501, 2019.

\bibitem{Weisser2019}
S.~Wei{\ss}er.
\newblock {\em BEM-based Finite Element Approaches on Polytopal Meshes}, volume
  130 of {\em Lecture Notes in Computational Science and Engineering}.
\newblock Springer International Publishing, 1 edition, 2019.

\bibitem{Weisser2018}
S.~Wei{\ss}er and T.~Wick.
\newblock The dual-weighted residual estimator realized on polygonal meshes.
\newblock {\em Comput. Methods Appl. Math.}, 18(4):753--776, 2018.

\bibitem{Wigley1964}
N.~M. Wigley.
\newblock Asymptotic expansions at a corner of solutions of mixed boundary
  value problems.
\newblock {\em J. Math. Mech.}, 13:549--576, 1964.

\bibitem{Wigley1969}
N.~M. Wigley.
\newblock On a method to subtract off a singularity at a corner for the
  {D}irichlet or {N}eumann problem.
\newblock {\em Math. Comp.}, 23:395--401, 1969.

\end{thebibliography}
\end{document}